\documentclass[final,onefignum,onetabnum]{siamart171218}
\usepackage{amssymb}
\usepackage{titlesec}
\usepackage{amsmath,amscd}
\usepackage{amsfonts}
\usepackage{graphicx}
\usepackage{ulem}
\usepackage{float}
\usepackage[caption=false]{subfig}
\usepackage{float} 
\usepackage{multirow}
\usepackage{diagbox}
\usepackage{mathrsfs}
\usepackage{indentfirst}
\usepackage{xcolor}
\usepackage{enumitem}
\usepackage{tikz}
\usepackage{etoolbox}
\usepackage{geometry}
\usepackage{fancyhdr}                                
\usepackage{lastpage}    
\usepackage{pifont}                                   
\usepackage{layout}
\usepackage{ulem}
\usepackage[justification=centering]{caption}

\newcommand{\circled}[2][]{\tikz[baseline=(char.base)]
    {\node[shape = circle, draw, inner sep = 1pt]
    (char) {\phantom{\ifblank{#1}{#2}{#1}}};%
    \node at (char.center) {\makebox[0pt][c]{#2}};}}
\robustify{\circled}

\usepackage{lipsum}
\usepackage{amsfonts}
\usepackage{graphicx}
\usepackage{epstopdf}
\usepackage{algorithmic}
\ifpdf
  \DeclareGraphicsExtensions{.eps,.pdf,.png,.jpg}
\else
  \DeclareGraphicsExtensions{.eps}
\fi

\newsiamremark{remark}{Remark}
\newsiamremark{hypothesis}{Hypothesis}
\crefname{hypothesis}{Hypothesis}{Hypotheses}
\newsiamthm{claim}{Claim}

\headers{Accelerated primal dual fixed point algorithm}{Yanan Zhu}

\title{Accelerated primal dual fixed point algorithm\thanks{Submitted to the editors DATE.}}
\author{Ya-Nan Zhu}
\usepackage{amsopn}

\ifpdf
\hypersetup{
  pdftitle={Accelerated primal dual fixed point algprithm},
  pdfauthor={Ya-Nan Zhu}
}
\fi

\begin{document}

\maketitle

\begin{abstract}
This work proposes an Accelerated Primal–Dual Fixed-Point (APDFP) method that employs Nesterov type acceleration to solve composite problems of the form $\min_x\, f(x)+g\circ B(x)$, where $g$ is nonsmooth and $B$ is a linear operator. The APDFP features fully decoupled iterations and can be regarded as a generalization of Nesterov’s accelerated gradient the setting where the $B$  can be non identity matrix.
Theoretically, we improve the convergence rate of the partial primal-dual gap with respect to the Lipschitz constant of gradient of $f$ from $\mathcal{O}(\frac{1}{k})$ to $\mathcal{O}(\frac{1}{k^2})$. Numerical experiments on graph-guided logistic regression and CT image reconstruction are conducted to validate the correctness and demonstrate the efficiency of the proposed method.
\end{abstract}


\section{Introduction.}
In this work, we consider the following structured convex optimization problem:
\begin{equation}\label{model}
\min_{x \in \mathbb{R}^d}~F(x) := f(x) + g\circ B(x),
\end{equation}
where $f: \mathbb{R}^d \to \mathbb{R}\cup \{+\infty\}$ is a proper smooth convex function with an $L_f$-Lipschitz continuous gradient, $g: \mathbb{R}^r \to \mathbb{R} \cup \{+\infty\}$ is a proper, convex, and lower semi-continuous (l.s.c.) function that may notbe differentiable, and $B \in \mathbb{R}^{r \times d}$ is a matrix. The composite nature of the objective function $F(x)$ makes it well-suited for a broad class of applications such as graph-guided fused lasso~\cite{GL1,GL2}, low-rank matrix completion~\cite{lowrank}, total variation-based image reconstruction~\cite{TV, CP, PDHG, PDHGm}, sparse coding~\cite{sc1,sc2}, compressed sensing~\cite{cs1,cs2,cs3}, etc.
Given the potentially large-scale nature of problem~\eqref{model}, efficient optimization methods are essential. In particular, first-order proximal splitting algorithms have gained widespread popularity due to their scalability, simplicity, and strong theoretical guarantees~\cite{Bookbeck,Booklan,condatreview,Booknesterov}. These methods are particularly appealing when the proximal operators of $f$ or $g$ (or their Fenchel conjugates) admit closed-form solutions or can be computed efficiently.Over the past decade, considerable progress has been made in designing first-order splitting algorithms for problem~\eqref{model}. Below, we briefly review some representative algorithms.
\subsection{$B = I$.}
We now turn to a special case of problem~\eqref{model}, namely the composite minimization problem
\begin{equation}\label{model2}
\min_{x \in \mathbb{R}^d}~F(x) := f(x) + g(x),
\end{equation}
which corresponds to setting $B = I$ in \textbf{(\ref{model})}.
A widely used approach for solving~\eqref{model2} is the Proximal Gradient Descent (PGD) method, also known as the proximal forward–backward splitting algorithm~\cite{PGD}. As summarized in Algorithm~\ref{PGD}, PGD updates the iterate by first performing a gradient descent step on $f$, followed by applying the proximity operator of $\gamma g$:
\begin{equation}
\mathrm{Prox}_{\gamma g}(s) = \arg\min_{y \in \mathbb{R}^d} \left\{ \gamma g(y) + \frac{1}{2}\|y - s\|_2^2 \right\},
\end{equation}
For the problems we are concern, we assume its proximity operator of $g$ has closed-form solution. 
\begin{algorithm*}[htbp]
\caption{\textbf{Proximal Gradient Descent (PGD)}}
\label{PGD}
\begin{algorithmic}[1]
\REQUIRE Initial point $x_1 \in \mathbb{R}^d$, step size $\gamma_k > 0$, and maximum iteration number $K$.
\FOR{$k = 1,2,\dots,K$}
\STATE $x_{k+1} = \mathrm{Prox}_{\gamma_k g}( x_k - \gamma_k\nabla f(x_k) )$
\ENDFOR
\ENSURE $x_{K + 1}$
\end{algorithmic}
\end{algorithm*}
Thanks to its simplicity and scalability, PGD has found widespread applications in large-scale data analysis, signal recovery, and image reconstruction. Theoretically, PGD has $\mathcal{O}(1/k)$ convergence rate in terms of the function value, where $k$ denotes the iteration counter. Nevertheless, in practice PGD often converges slowly, particularly for ill-conditioned problems. This drawback has motivated the development of various acceleration techniques.
One of the most influential acceleration methods is the Fast Iterative Shrinkage-Thresholding Algorithm (FISTA, Algorithm~\ref{FISTA}) proposed by Beck and Teboulle~\cite{FISTA}, which is inspired by Nesterov’s seminal work on accelerated gradient methods~\cite{Nes83}. Compared to PGD, FISTA introduces an extrapolation step that leverages both the current and previous iterates. This seemingly simple modification improves the convergence rate of PGD to $\mathcal{O}(1/k^2)$ in terms of the function value, which is optimal for smooth convex problems, and has made FISTA one of the most widely adopted first-order methods in large-scale optimization.
\begin{algorithm*}[htbp]
\caption{\textbf{Fast Iterative Shrinkage-Thresholding Algorithm (FISTA)}}
\label{FISTA}
\begin{algorithmic}[1]
\REQUIRE Initial points $x_0 = x_1 \in \mathbb{R}^d$, step size $\gamma_k > 0$, and maximum iteration number $K$.
\FOR{$k =1,2,\dots,K$}
\STATE $y_k = x_k + \frac{k - 1}{k + 2}(x_k - x_{k-1})$
\STATE $x_{k+1} = \mathrm{Prox}_{\gamma_k g}(y_k - \gamma_k\nabla f(y_k))$
\ENDFOR
\ENSURE $x_{K + 1}$
\end{algorithmic}
\end{algorithm*}
In parallel, Nesterov introduced an alternative acceleration framework (Algorithm~\ref{nes2nd}), commonly referred to as Nesterov’s Accelerated Gradient (Algorithm~\ref{nes2nd})(NAG)~\cite{nesterov1,nesterov2,Nes2nd}. Compared to FISTA, which extrapolates iterates through a single momentum term, NAG maintains two coupled sequences: a primary sequence updated by gradient steps, and an aggregated sequence (the superscript “ag” is used to denote this aggregated sequence) that serves as the search point for gradient evaluation. This sequence construction provides a principled mechanism for generating momentum and yields the optimal $\mathcal{O}(1/k^2)$ convergence rate as FISTA.
\begin{algorithm*}[htbp]
\caption{\textbf{Nesterov's Accelerated Gradient (NAG)}}
\label{nes2nd}
\begin{algorithmic}[1]
\REQUIRE Initial point $x_1 = x_1^{\text{ag}} \in \mathbb{R}^d$, parameters $\gamma_k > 0$, and maximum iteration number $K$.
\FOR{$k = 1,2,\dots,K$}
\STATE $\theta_k = \frac{2}{k + 1}$
\STATE $x_k^{\text{md}} = (1 - \theta_k) x_k^{\text{ag}} + \theta_k x_k$
\STATE $x_{k+1} = \mathrm{Prox}_{\frac{\gamma_k}{\theta_k} g}\left(x_k - \frac{\gamma_k}{\theta_k} \nabla f(x_k^{\text{md}})\right)$
\STATE $x_{k+1}^{\text{ag}} = (1 - \theta_k) x_k^{\text{ag}} + \theta_k x_{k+1}$
\ENDFOR
\ENSURE $x_{K + 1}^{\text{ag}}$
\end{algorithmic}
\end{algorithm*}
\subsection{$B \not = I$.}
When the matrix \( B \) is not the identity, directly applying above algorithms encounters the difficulty of evaluating the proximal operator of the composite function \( g \circ B(\cdot) \). While the proximity operator of \( g \) alone is often admits a closed-form solution, that of the composition \( g \circ B(\cdot) \) is generally not easy and typically necessitates a specialized subroutine. In the following, we summarize several existing strategies to circumvent this problem.
\subsubsection{Primal Dual Hybrid Gradient.}
To bypass the annoying matrix \( B \), one can reformulate problem \textbf{(\ref{model})} as the following saddle-point (min-max) problem:
\begin{equation}\label{minmax}
\min_{x \in \mathbb{R}^d} \max_{y \in Y}~f(x) + \langle Bx, y \rangle - g^*(y),    
\end{equation}
where \( g^* \) denotes the convex conjugate of \( g \), defined on domain $Y$. (Here, we omit some technical assumptions under which problems \textbf{(\ref{model})} and \textbf{(\ref{minmax})} are equivalent; interested readers are referred to \cite{Monotone} for details.)
Under this min-max reformulation, the challenge of evaluating the proximity operator of the composite function \( g \circ B \) is transformed into computing the proximal operator of \( g^* \). Since the proximity operator of \( g \) often admits a closed-form expression, the same typically holds for its convex conjugate \( g^* \), via the application of the Moreau decomposition \cite{Monotone}. One of the most widely used algorithms for solving \textbf{(\ref{minmax})} is the modified Primal-Dual Hybrid Gradient method (PDHGm, Algorithm~\ref{PDHGm}) \cite{PDHG,OriPDHG,PDHGm}, also known as the Chambolle–Pock algorithm (CP) \cite{CP}. PDHGm proceeds by alternating between evaluating the proximity operators of \( f \) and \( g^* \), followed by an extrapolation step on the primal variable \( x \).
Even though PDHGm circumvents the difficulty associated with the proximity operator of \( g \circ B \), computing the proximity operator of \( f \) remains challenging in many practical scenarios. To address this issue, one can exploit the smoothness of \( f \) and fully decouple the iterations by linearizing \( f \), i.e.,
\begin{equation}
\begin{aligned}
x_{k + 1} & = \mathrm{Prox}_{\tau_k f}(x_k - \tau_k B^Ty_{k + 1}) \\   
& = \arg\min_{x \in \mathbb{R}^d}~f(x) + \frac{1}{2\tau_k}\|x - (x_{k} - \tau_k B^Ty_{k + 1})\|_2^2 \\
& = \arg\min_{x \in \mathbb{R}^d}~f(x) + \langle x, B^Ty_{k + 1}\rangle + \frac{1}{2\tau_k} \|x - x_k\|_2^2 \\
Linearization & \Rightarrow \\
x_{k + 1}
& = \arg\min_{x \in \mathbb{R}^r}~f(x_k) + \langle \nabla f(x_k), x - x_k\rangle + \frac{1}{2\eta_k} \|x - x_k\|_2^2 + \langle x, B^Ty_{k + 1}\rangle  \\
& \qquad + \frac{1}{2\tau_k} \|x - x_k\|_2^2 \\
\Rightarrow
x_{k + 1} & = x_k - \gamma_k \nabla f(x_k) - \gamma_k B^T y_{k + 1}
\end{aligned}    
\end{equation}
where $\gamma_k = \frac{\eta_k\tau_k}{\eta_k + \tau_k}$.
\begin{algorithm*}[htbp]
\caption{\textbf{Modified Primal Dual Hybrid Gradient (PDHGm/CP)}}
\label{PDHGm}
\begin{algorithmic}[1]
\REQUIRE Choose $x_1,\overline{x}_1 \in \mathbb{R}^d$, $y_1 \in \mathbb{R}^r$ and proper positive parameters  $\sigma_k, \tau_k, \alpha_k$.
\STATE \textbf{For} $k = 1,2,\cdots,K$ 
\STATE \qquad $y_{k + 1} = \mathrm{Prox}_{\sigma_k g^*}(\sigma_kB\overline{x}_k + y_k)$
\STATE \qquad $x_{k + 1} = \mathrm{Prox}_{\tau_k f}(x_k - \tau_k B^Ty_{k + 1})$ 
\STATE \qquad $\overline{x}_{k+1} = x_{k + 1} + \alpha_k(x_{k + 1} - x_k)$
\STATE \textbf{End For} 
\ENSURE $x_{K + 1}$.
\end{algorithmic}
\end{algorithm*}
The resulting algorithm is known as the linearized PDHGm (LPDHGm, Algorithm~\ref{LPDHGm}) \cite{APD}. It is not difficult to establish an ergodic convergence rate of \(\mathcal{O}(L_f/k + \|B\|_2/k)\) for the partial primal-dual gap of the LPDHGm \cite{APD}. Although this rate matches that of the original PDHGm, its dependence on \(L_f\) is suboptimal. This suboptimal rate arises from the fact that the update of $x_{k + 1}$ in LPDHGm essentially performs a single PGD step. To improve the rate, it is natural to incorporate Nesterov’s acceleration into $x_{k + 1}$-update, thereby improving the convergence rate from $\mathcal{O}(1/k)$ to $\mathcal{O}(1/k^2)$. Such an improvement is particularly beneficial in ill-posed problems frequently encountered data science and imaging applications, where the Lipschitz constant \(L_f\) is often much larger than \(\|B\|_2\). These considerations motivate the development of an accelerated variant of LPDHGm, termed APD (Algorithm~\ref{APD},\cite{APD}). Theoretically, APD achieves an enhanced convergence rate of \(\mathcal{O}(L_f/k^2 + \|B\|_2/k)\).
\begin{algorithm*}
\caption{\textbf{Linearized Modified Primal Dual Hybrid Gradient (LPDHGm)}}
\label{LPDHGm}
\begin{algorithmic}[1]
\REQUIRE Choose $x_1,\overline{x}_1 \in \mathbb{R}^d$, $y_1 \in \mathbb{R}^r$ and proper positive parameters  $0 < \alpha_k = \sigma_{k - 1}/\sigma_k = \gamma_{k - 1}/\gamma_{k} \leq 1,~L_f\sigma_k + \|B \|_2^2\gamma_k\sigma_k \leq 1 $.
\STATE \textbf{For} $k = 1,2,\cdots,K$ 
\STATE \qquad $y_{k + 1} = \mathrm{Prox}_{\sigma_k g^*}(\sigma_kB\overline{x}_k + y_k)$
\STATE \qquad $x_{k + 1} = x_k - \gamma_k\nabla f(x_k) - \gamma_k B^Ty_{k + 1}$
\STATE \qquad $\overline{x}_{k+1} = x_{k + 1} + \alpha_{k + 1}(x_{k + 1} - x_k)$
% \STATE \qquad $y_{k+1}^{\text{ag}} = (1 - \theta_k) y_k^{\text{ag}} + \theta_k y_{k+1}$,
\STATE \textbf{End For} 
\ENSURE $x_{K + 1}$.
\end{algorithmic}
\end{algorithm*}
\begin{algorithm*}
\caption{\textbf{Accelerated linearized modified Primal Dual Hybrid Gradient (APD)}}
\label{APD}
\begin{algorithmic}[1]
\REQUIRE Choose $x_1 = x_1^{\text{ag}},\overline{x}_1 \in \mathbb{R}^d$, $y_1 \in \mathbb{R}^r$ and proper positive parameters $C > 0, \tau_k = C/\|B\|_2,\gamma_k = k/(2L_f + k\|B\|C), \alpha_{k + 1} = k/(k + 1)$.
\STATE \textbf{For} $k = 1,2,\cdots,K$ 
\STATE \qquad $\theta_k = \frac{2}{k + 1}$
\STATE \qquad $x_k^{\text{md}} = (1 - \theta_k) x_k^{\text{ag}} + \theta_k x_k$
% \STATE \qquad $\overline{x}_{k} = x_k - \frac{\gamma}{\theta_k}\nabla f(x_k^{\text{md}}) - \frac{\gamma}{\theta_k} B^Ty_k$, 
\STATE \qquad $y_{k+1} = \mathrm{Prox}_{\tau_k g^*}(\tau_kB\overline{x}_k + y_k)$
\STATE \qquad $x_{k + 1} = x_k - \gamma_k\nabla f(x_k^{\text{md}}) - \gamma_k B^Ty_{k + 1}$
\STATE \qquad $x_{k+1}^{\text{ag}} = (1 - \theta_k) x_k^{\text{ag}} + \theta_k x_{k+1}$
\STATE \qquad $\overline{x}_{k+1} = x_{k + 1} + \alpha_{k + 1}(x_{k + 1} - x_k)$
\STATE \textbf{End For} 
\ENSURE $x_{K + 1}^{ag}$.
\end{algorithmic}
\end{algorithm*}

\subsubsection{Alternating Direction Method of Multipliers.}
Another approach to mitigate the difficulty of computing the proximity operator of $g\circ B(\cdot)$ is to employ the Alternating Direction Method of Multipliers (ADMM) \cite{ADMM1,ADMM2,ADMM3,goldstein2009split}. To enable the application of ADMM, problem \textbf{(\ref{model})} is reformulated as the following equality-constrained optimization problem:
\begin{equation}\label{equality}
\begin{aligned}
& \min_{x \in \mathbb{R}^d}~f(x) + g(z),  \\  
& s.t.,~Bx = z
\end{aligned}
\end{equation}
ADMM solves the reformulated problem \textbf{(\ref{equality})} by performing alternating minimization of the following augmented Lagrangian with respect to the primal and dual variables (Algorithm~\ref{ADMM}).
\begin{equation}
L(x,z,y) = f(x) + g(z) + \langle Bx - z, y\rangle + \frac{\rho}{2}\|Bx - z\|_2^2  
\end{equation}
where $\rho > 0$ is an algorithm parameter. 
\begin{algorithm*}
\caption{\textbf{Alternating Direction Methods of Multiplier (ADMM)}}
\label{ADMM}
\begin{algorithmic}[1]
\REQUIRE Choose $x_1, z_1 = Bx_1 \in \mathbb{R}^d,y_1$, and positive parameters $\rho$.
\STATE \textbf{For} $k = 1,2,\cdots,K$ 
\STATE \qquad $x_{k + 1} = \arg\min_{x \in \mathbb{R}^d} f(x) + \langle Bx, y_k\rangle + \frac{\rho}{2}\|Bx - z_k\|_2^2 $
\STATE \qquad $z_{k + 1} = \mathrm{Prox}_{ g/\rho}(Bx_{k + 1} + y_k/\rho_k)$
\STATE \qquad $y_{k+1} = x_{k + 1} + \rho(Bx_{k + 1} - z_{k + 1})$
\STATE \textbf{End For} 
\ENSURE $x_{K + 1}$.
\end{algorithmic}
\end{algorithm*}
It is worth noting that the first subproblem in ADMM is generally difficult to solve in practice. Similar to the LPDHGm, one can perform linearization on augmented Lagrangian to decouple the iterations. However, unlike LPDHGm, the ADMM framework offers three distinct linearization strategies, each leading to different algorithmic variants:
\begin{enumerate}
    \item Linearize function $f(x)$: Linearize function $f$ will transfer the subproblem to the following 
    \begin{equation}\label{LADMMeq}
    \begin{aligned}    
    & x_{k + 1}  = \arg\min_x ~f(x) + \langle Bx, y_k\rangle + \frac{\rho}{2}\|Bx - z_k\|_2^2 \\
    & Linearize~ f(x) \Rightarrow \\
    & x_{k + 1}  = \arg\min_x ~ f(x_k) + \langle \nabla f(x_k),x - x_k\rangle + \frac{1}{2\gamma_k}\|x - x_k\|_2^2 + \langle Bx, y_k\rangle + \frac{\rho}{2}\|Bx - z_k\|_2^2 \\
    \Leftrightarrow
    & (I + \gamma_k\rho B^TB)x_{k + 1} = x_k - \gamma_k (\nabla f(x_k) + B^Ty_k - \rho B^Tz_k) 
    \end{aligned}
    \end{equation}
    Replace the first subproblem by \textbf{(\ref{LADMMeq})}, the resulting algorithm is called linearized ADMM (LADMM) \cite{LADMM1,LADMM2,LADMM3}. It is observed that the LADMM still need solving a linear system of equations.
    \item Linearize augmented term $\frac{\rho}{2}\|Bx - z\|_2^2$: Linearize $\frac{\rho}{2}\|Bx - z\|_2^2$ will transfer the subproblem to the following 
    \begin{equation}\label{PADMMeq}
    \begin{aligned} 
     x_{k + 1}  &= \arg\min_x ~f(x) + \langle Bx, y_k\rangle + \frac{\rho}{2}\|Bx - z_k\|_2^2 \\
    & Linearize~ \frac{\rho}{2}\|Bx - z\|_2^2 \Rightarrow \\
    x_{k + 1} & = \arg\min_x ~  f(x) + \langle Bx, y_k\rangle + \rho_k\langle Bx_k - z_k, Bx\rangle  + \frac{1}{2\gamma_k}\|x - x_k\|_2^2 \\
    & = \arg\min_x ~  f(x) + \langle Bx, y_k\rangle + \rho_k\langle Bx - z_k, Bx\rangle + \rho_k\langle Bx_k - Bx, Bx\rangle + \frac{1}{2\gamma_k}\|x - x_k\|_2^2 \\
    & = \arg\min_x ~  f(x) + \langle Bx, y_k\rangle + \frac{\rho_k}{2}\| Bx - z_k\|_2^2  + \frac{1}{2\gamma_k}\|x - x_k\|_{I - \gamma_k\rho_k BB^T}^2 \\
    & = \mathrm{Prox}_{\gamma_kf}(x_k - \gamma_k(B^Ty_k + \rho_k B^T(Bx_k - z_k)))
    \end{aligned}
    \end{equation}
    It is observed from the third equation of \textbf{(\ref{PADMMeq})} that linearizing the augmented term is equivalent to add a quadratic term with generalized norm induced by the matrix $I - \gamma_k\rho_k BB^T$. The technique is called the precondition and the resulting algorithm is called preconditioned ADMM (PADMM) \cite{CP,APD}. Generally, the proximity operator of $f$ is not easy, and the update of $x_{k + 1}$ still require solving subproblem.
    \item Linearize $f(x) + \frac{\rho}{2}\|Bx - z\|_2^2$: Even applied the linearization, the LADMM and PADMM still need subproblem solving. To fully decouple the iteration, one can  linearize both $f$ and $\frac{\rho}{2}\|Bx - z\|_2^2$, this gives
    \begin{equation*}
    \begin{aligned}
    x_{k + 1} & = \arg\min_x ~ \langle \nabla f(x_k), x - x_k \rangle + \rho_k\langle Bx_k - z_k, Bx\rangle + \frac{1}{2\gamma_k}\|x - x_k\|_2^2 + \langle Bx, y_k\rangle \\
\Leftrightarrow    x_{k + 1} &= x_k - \gamma_k(\nabla f(x_k) + \rho_k(B^T(Bx_k - z_k)) + B^Ty_k) 
    \end{aligned}
    \end{equation*}
    The resulting algorithm is called Linearized  Precondition ADMM (LP-ADMM) 
\end{enumerate}

\begin{algorithm*}
\caption{\textbf{Linearized Preconditioned  ADMM (LP-ADMM)}}
\label{LPADMM}
\begin{algorithmic}[1]
\REQUIRE Choose $x_1,y_1, z_1 = Bx_1 \in \mathbb{R}^d$ and proper positive $C > 0$. Let $\rho = C/\|B\|_2, \gamma = 1/(L_{f} + \rho \lVert B\rVert^2)$.
\STATE \textbf{For} $k = 1,2,\cdots,K$ 
\STATE \qquad $x_{k + 1} = x_k - \gamma(\nabla f(x_k) + \rho(B^T(Bx_k - z_k)) + B^Ty_k) $ 
\STATE \qquad $z_{k + 1} = \mathrm{Prox}_{ g/\rho}(Bx_{k + 1} + y_k/\rho_k)$
\STATE \qquad $y_{k+1} = x_{k + 1} + \rho(Bx_{k + 1} - z_{k + 1})$
\STATE \textbf{End For} 
\ENSURE $x_{K + 1}$.
\end{algorithmic}
\end{algorithm*}
In contrast to LADMM and PADMM, the LP-ADMM fully decouples the iteration without requiring any inner subroutines, thereby offering improved computational efficiency. The ergordic convergence rate of LP-ADMM has been established as \(\mathcal{O}(L_f/k + \|B\|_2/k)\) in terms of the objective error \(F(\overline{x}_k) - F(x^*)\) \cite{AADMM} ($\overline{x}_k = \frac{1}{k}\sum_{i = 1}^k x_i$). However, analogous to LPDHGm, this rate is suboptimal with respect to the Lipschitz constant \(L_f\). To improve the convergence rate, \cite{AADMM} incorporates Nesterov’s acceleration into the update of the primal variable \(x\), yielding an improved convergence rate of \(\mathcal{O}(L_f/k^2 + \|B\|_2/k)\) (in terms of aggregate iterate $x_k^{\text{ag}}$). The resulting algorithm, termed Accelerated ADMM (AADMM) (Algorithm \ref{AADMM}), achieves a provably faster convergence rate while maintaining the fully explicit and decoupled iteration structure of LP-ADMM.
\begin{algorithm*}
\caption{\textbf{Accelerated ADMM (AADMM)}}
\label{AADMM}
\begin{algorithmic}[1]
\REQUIRE Choose $x_1, y_1, z_1 = Bx_1 \in \mathbb{R}^d$, and proper positive $\rho > 0$. Let $\sigma_k = (k - 1)\rho/k, \gamma_k = k/(2/L_{f} + \rho k\|B\|_2^2)$.
\STATE \textbf{For} $k = 1,2,\cdots,K$ 
\STATE \qquad $\theta_k = \frac{2}{k + 1}$
\STATE \qquad $x_k^{\text{md}} = (1 - \theta_k) x_k^{\text{ag}} + \theta_k x_k$
% \STATE \qquad $\overline{x}_{k} = x_k - \frac{\gamma}{\theta_k}\nabla f(x_k^{\text{md}}) - \frac{\gamma}{\theta_k} B^Ty_k$, 
\STATE \qquad $x_{k + 1} = x_k - \gamma_k(\nabla f(x_k) + \sigma_k(B^T(Bx_k - z_k)) + B^Ty_k)$ 
\STATE \qquad $x_{k+1}^{\text{ag}} = (1 - \theta_k) x_k^{\text{ag}} + \theta_k x_{k+1}$ 
\STATE \qquad $z_{k + 1} = \mathrm{Prox}_{ g/\rho}(Bx_{k + 1} + y_k/\rho_k)$
\STATE \qquad $z_{k+1}^{\text{ag}} = (1 - \theta_k) z_k^{\text{ag}} + \theta_k z_{k+1}$ 
\STATE \qquad $y_{k+1} = x_{k + 1} + \rho(Bx_{k + 1} - z_{k + 1})$
% \STATE \qquad $y_{k+1}^{\text{ag}} = (1 - \theta_k) y_k^{\text{ag}} + \theta_k y_{k+1}$,
\STATE \textbf{End For} 
\ENSURE $x_{K + 1}^{ag}$.
\end{algorithmic}
\end{algorithm*}

\subsubsection{Primal Dual Fixed Point.}
In the previous two subsections, we introduced two approaches to circumvent the computation of the proximity operator of $g \circ B(\cdot)$, namely PDHGm and ADMM, along with their accelerated linearized variants, APD and AADMM. Although these methods benefit from decoupled update rules and improved convergence rates, the linearization introduces additional algorithmic parameters. To guarantee convergence, these parameters must satisfy certain complex conditions, thereby reintroducing a degree of parameter tuning burden in practical applications (APD need to tune constant $C$ and AADMM need to tune $\rho$).  In this subsection, we turn our attention to the primary algorithm of this work—the Primal-Dual Fixed Point (PDFP) method. The derivation of PDFP is similar to that of the PGD, which performs a one-step gradient descent on the function $f$. The key difference lies in the approximation of the proximity operator $\mathrm{Prox}_{\gamma g \circ B}(\cdot)$ using a different algorithm. Specifically, according to the definition of the proximity operator, for a given vector $y$, the operator $\mathrm{Prox}_{\gamma g \circ B}(y)$ is defined as
\begin{equation}\label{prox}
\mathrm{Prox}_{\gamma g\circ B}(y) = \arg\min_{v} \gamma g\circ B(v) + \frac{1}{2}\|v - y\|_2^2  
\end{equation}
At first glance, the difficulty of \textbf{(\ref{prox})} appears similar to that of \textbf{(\ref{model})}. However, the smooth term in \textbf{(\ref{prox})} has the specific form of a squared norm, which makes it easier to design iterative algorithms for solving it. For example, \cite{FP2O} proposed a fixed-point iterative method called Fixed-Point algorithms based on the Proximity Operator (FP${}^2$O) solving \textbf{ (\ref{prox})} as follows:
\begin{itemize}
    \item Performing the following fixed point iteration to obtain the fixed point $v^*$ 
    \begin{equation}
     v_{k + 1} = (I - \mathrm{Prox}_{\frac{\gamma}{\lambda} g})(By + (I - \lambda BB^T)v_k),k = 1,2,\cdots   
    \end{equation}
    where $\lambda$ is a parameter satisfying $0 < \lambda < \frac{1}{\rho_{\max}(BB^T)}$ (here, $\rho_{\max}$ denotes the maximum eigenvalue of the matrix $BB^T$). 
    \item The $\mathrm{Prox}_{\gamma g\circ B}(y)$ is computed by $\mathrm{Prox}_{g\circ B}(y) = y - \lambda B^Tv^*$.
\end{itemize}
If we directly apply FP${}^2$O to compute $\mathrm{Prox}_{\gamma g \circ B}(y)$, it still involves solving a subproblem. To fully decouple the algorithm, it is intuitively natural to perform only a single iteration of FP${}^2$O to approximate $\mathrm{Prox}_{\gamma  g \circ B}(y)$. By combining one step of gradient descent on $f$ with one step of FP${}^2$O for approximating $\mathrm{Prox}_{ \gamma g \circ B}(y)$, we arrive at the PDFP (a.k.a PDFP2O) method \cite{PDFP}.
\begin{algorithm*}
\caption{\textbf{Primal Dual Fixed Point (PDFP)}}
\label{PDFP1}
\begin{algorithmic}[1]
\REQUIRE Choose $x_1,y_1$ and proper positive parameters  $\lambda, \gamma$
\STATE \textbf{For} $k = 1,2,\cdots,K$ 
\STATE \qquad $\overline{x}_{k} = x_k - \gamma\nabla f(x_k)$ 
\STATE \qquad $y_{k+1} = (I - \mathrm{Prox}_{\frac{\gamma}{\lambda}g})(B\overline{x}_{k} + (I - \lambda BB^T)y_k)$ 
\STATE \qquad $x_{k + 1} = \overline{x}_{k} - \lambda B^Ty_{k + 1}$ 
\STATE \textbf{End For} 
\ENSURE $x_{K + 1}$.
\end{algorithmic}
\end{algorithm*}
The PDFP algorithm converges when the step size satisfies $0 < \gamma \leq \frac{2}{L_f}$ and $0 < \lambda < \frac{1}{\rho_{\max}(BB^T)}$. It is observed that the conditions for these two parameters are independently related to the model parameters and do not require satisfying any complex interdependence. In addition, numerical experiments show that the performance of PDFP is relatively insensitive to the choice of $\lambda$. In practice, it is often set to its theoretical upper bound $\frac{1}{\rho_{\max}(BB^T)}$. The only parameter that significantly affects the convergence speed is the step size $\gamma$.  Analogous to gradient-based methods, setting $\gamma = \frac{2}{L_f}$ often yields promising results. This is particularly advantageous in practical applications, as once model \textbf{(\ref{model})} is defined, the optimal parameters for the algorithm can be directly specified without the need for burdensome parameter tuning. Furthermore, when $B = I$ and $\lambda$ is chosen as $1$, it is readily to verify that PDFP can be rewritten as 
\begin{equation}
\begin{aligned}
y_{k + 1} & = (I - \mathrm{Prox}_{\gamma g})(x_k - \gamma \nabla f(x_k))) \\
x_{k + 1} & = x - \gamma \nabla f(x_k) - y_{k + 1} \\
& = \mathrm{Prox}_{\gamma g}(x_k - \gamma \nabla f(x_k))
\end{aligned}
\end{equation}
This is essentially the PGD! Therefore, the PDFP method can be viewed as a generalization of PGD, and in constast to directly apply PGD to model \textbf{(\ref{model})}, PDFP does not require any subproblem solving. By employing the Moreau decomposition \cite{cvxhl} in the update of dual variable $y_{k + 1}$, the update rule of PDFP can be reformulated using the proximity operator of the conjugate function $g^*$ \cite{PDFP, SVRGPDFP} (see Algorithm~\ref{PDFP}). This resemblance recalls the minmax formulation of PDHG discussed earlier. In fact, it has been shown in \cite{SVRGPDFP} that, similar to PDHG, the PDFP method also solves the min-max problem \textbf{(\ref{minmax})}, and the update rule in Algorithm \ref{PDFP} is also called Loris–Verhoeven \cite{LR} or Proximal Alternating Predictor–Corrector (PAPC) \cite{PAPC} algorithm.
\begin{algorithm*}
\caption{\textbf{Primal Dual Fixed Point Method (PDFP/Loris–Verhoeven/PAPC)}}
\label{PDFP}
\begin{algorithmic}[1]
\REQUIRE Choose $x_1, y_1 $ and proper $\lambda, \gamma > 0$.
\STATE \textbf{For} $k = 1,2,\cdots,K$ 
\STATE \qquad $\overline{x}_{k} = x_k - \gamma\nabla f(x_k) -\gamma B^Ty_k$ 
\STATE \qquad $y_{k+1} =\mathrm{Prox}_{\frac{\lambda}{\gamma}g^*}(\frac{\lambda}{\gamma}B\overline{x}_{k} + y_k)$ 
\STATE \qquad $x_{k + 1} = x_k - \gamma\nabla f(x_k) -\gamma B^Ty_{k + 1}$  
\STATE \textbf{End For} 
\ENSURE $x_{K + 1}$.
\end{algorithmic}
\end{algorithm*}
Although the original PDFP paper does not provide a convergence rate under general convexity assumptions, it is not very difficult to prove its ergodic convergence rate of $\mathcal{O}\left(\frac{L_f + \rho_{\max}(I -\lambda BB^T)}{k}\right)$ in terms of the partial primal-dual gap \cite{SVRGPDFP}( provided that $0 < \gamma \leq \frac{1}{L_f}$ and $0 < \lambda \leq \frac{1}{\rho_{\max}(BB^T)}$). However, this rate is suboptimal with respect to its dependence on $L_f$. To improve the convergence speed, \cite{IPDFP} proposes an inertial version of PDFP (IPDFP, see Algorithm \ref{iPDFP}). However, it only guarantees convergence without providing explicit theoretical rates, and the inertial parameter $\alpha_k$ is typically problem-dependent and requires tuning.
\begin{algorithm*}
\caption{\textbf{Inertial Primal Dual Fixed Point Method (IPDFP)}}
\label{iPDFP}
\begin{algorithmic}[1]
\REQUIRE Choose $x_0 = x_1, y_0 = y_1$ and proper positive parameters $\alpha_k,\lambda, \gamma > 0$.
\STATE \textbf{For} $k = 1,2,\cdots,K$ 
\STATE \qquad $z_{k} = x_k + \alpha_k (x_k - x_{k -1})$ 
\STATE \qquad $v_{k} = y_k + \alpha_k (I - \lambda BB^T)(y_k - y_{k -1})$ 
\STATE \qquad $\overline{x}_{k} = z_k - \gamma\nabla f(z_k) -\gamma B^Ty_k$ 
\STATE \qquad $y_{k+1} =\mathrm{Prox}_{\frac{\lambda}{\gamma}g^*}(\frac{\lambda}{\gamma}B\overline{x}_{k} + v_k)$ 
\STATE \qquad $x_{k + 1} = x_k - \gamma\nabla f(x_k) -\gamma B^Ty_{k + 1}$  
\STATE \textbf{End For} 
\ENSURE $x_{K + 1}$.
\end{algorithmic}
\end{algorithm*}
Analogous to the developments of APD and AADMM, this motivates the application of Nesterov-type acceleration techniques to enhance the convergence behavior of PDFP. Since PDFP can be regarded as a generalization of PGD, and is further characterized by its simple update rule and minimal parameter tuning requirements, the study of its accelerated variants is expected not only to unify and extend existing algorithmic frameworks, but also to provide efficient and scalable solutions to a broader class of problems that arise in data science and imaging applications.

\subsection{Organization of the paper.}
The remainder of this paper is organized as follows. Section~\ref{sec:algorithm} introduces the update rule of the APDFP algorithm. In Section~\ref{sec:convergence}, we establish the convergence analysis, presenting the main theorem and deferring the detailed proofs to Appendix~\ref{sec:appendix}. Section~\ref{sec:experiments} reports numerical experiments on graph-guided logistic regression and CT image reconstruction, which validate the correctness and efficiency of APDFP. The paper is concluded in Section~\ref{sec:conclusion}.

\section{Algorithm.}\label{sec:algorithm}
In this section, we present the detailed update rules of the APDFP (Algorithm~\ref{APDFP}). Similar to APD and AADMM, APDFP maintains both the generic iterates $x_k, y_k$ and the aggregated iterates $x_k^{\text{ag}}, y_k^{\text{ag}}$. The iteration begins by computing a convex combination of $x_k$ and $x_k^{\text{ag}}$ using the parameter $\theta_k$, yielding an intermediate point $x_k^{\text{md}}$(here we follow the notation in~\cite{APD}). 
Based on $x_k^{\text{md}}$, a single PDFP step is performed to update the primal and dual variables $x_{k+1}$ and $y_{k+1}$. It is worth noting that, in this update, the step size $\gamma$ in the original PDFP is replaced by $\gamma_k / \theta_k$, and the parameter $\frac{\lambda}{\gamma}$ is replaced by $\frac{\lambda}{\gamma_k}\theta_k$.
Subsequently, the aggregated iterates $x_{k+1}^{\text{ag}}$ and $y_{k+1}^{\text{ag}}$ are updated using the newly computed primal and dual variables. The aggregated dual iterate $y_{k+1}^{\text{ag}}$, although not participate in the iteration process, is used in the convergence analysis. \\
\begin{remark}
\textbf{(Connection with other algorithms).} As previously mentioned, PDFP reduces to the PGD when $B = I$ and $\lambda = 1$. A natural question then arises: what does APDFP reduce to under the same conditions? To answer this, we perform the following calculation. When $B = I$ and $\lambda = 1$, the update for $y_{k+1}$ in the APDFP step of Algorithm~\ref{APDFP} simplifies as follows:
\begin{equation}\label{APDFP_dual}
\begin{aligned}    
y_{k+1} & = \mathrm{Prox}_{\frac{\theta_k}{\gamma_k}g^*}\left(\frac{\theta_k}{\gamma_k}\overline{x}_k + y_k\right)\\
& = \mathrm{Prox}_{\frac{\theta_k}{\gamma_k}g^*}\Big(\frac{\theta_k}{\gamma_k}\big(x_k - \frac{\gamma_k}{\theta_k}\nabla f(x_k^{\text{md}}) - \frac{\gamma_k}{\theta_k} y_k \big) + y_k\Big)\\
& = \mathrm{Prox}_{\frac{\theta_k}{\gamma_k}g^*}\left(\frac{\theta_k}{\gamma_k}\big(x_k - \frac{\gamma_k}{\theta_k}\nabla f(x_k^{\text{md}})\big)\right)\\
\end{aligned}    
\end{equation}
Plug Eq. \textbf{(\ref{APDFP_dual})} into the $x_{k + 1}$ update yields
\begin{equation}\label{APDFP_primal}
\begin{aligned}    
x_{k + 1} & = x_k - \frac{\gamma_k}{\theta_k}\nabla f(x_k^{\text{md}}) - \mathrm{Prox}_{\frac{\theta_k}{\gamma_k}g^*}\left(\frac{\theta_k}{\gamma_k}(x_k - \frac{\gamma_k}{\theta_k}\nabla f(x_k^{\text{md}}))\right) \\
& = \mathrm{Prox}_{\frac{\gamma_k}{\theta_k}g}\left(x_k - \frac{\gamma_k}{\theta_k}\nabla f(x_k^{\text{md}})\right)
\end{aligned}    
\end{equation}
where the second equality uses the Moreau decomposition \cite{cvxhl,PGD,PDHGm}. \\
By combining Steps 2 and 6 of Algorithm~\ref{APDFP}, and noting that $y_{k+1}^{\textit{ag}}$ does not participate in the iteration process, we arrive at the NAG (when the $\theta_k = \frac{2}{k + 1}$, as shown in Corollary \ref{col1}). Therefore, APDFP can be regarded as a generalization of NAG in the presence of a general linear operator $B$.
In contrast to approaches that directly apply NAG to solve problem~\textbf{(\ref{model})}, APDFP avoids the need for computing the proximity operator of the composite function $g \circ B$. Furthermore, it is straightforward to verify that APDFP reduces to the original PDFP algorithm when $\theta_k := 1, \gamma_k := \gamma$. \\
\end{remark}
\begin{remark}
\textbf{(Comparison to APD).} As discussed previously, PDFP addresses the same min-max reformulation as LPDHGm. Therefore, it is natural to compare their respective accelerated variants. From Algorithm~\ref{APD}, we observe that the updates in APDFP are analogous to those in APD, except for the additional extrapolation step (Step 6 in Algorithm~\ref{APD}). Another key distinction lies in the explicit presence of the parameter $\theta_k$ in the update rules for both the primal and dual variables in APDFP.
\begin{algorithm*}
\caption{\textbf{Accelerated Primal Dual Fixed Point Method (APDFP)}}
\label{APDFP}
\begin{algorithmic}[1]
\REQUIRE Choose $x_1 = x_1^{\text{ag}}$, $y_1 = y_1^{\text{ag}}$, $\bar{x}_1 = y_1$ and proper positive parameters $\theta_k,\gamma_k,\lambda$.
\STATE \textbf{For} $k = 1,2,\cdots,K - 1$ 
\STATE \qquad $x_k^{\text{md}} = (1 - \theta_k) x_k^{\text{ag}} + \theta_k x_k$ 
\STATE \qquad $\overline{x}_{k} = x_k - \frac{\gamma_k}{\theta_k}\nabla f(x_k^{\text{md}}) - \frac{\gamma_k}{\theta_k} B^Ty_k$ 
\STATE \qquad $y_{k+1} = \mathrm{Prox}_{\frac{\lambda}{\gamma_k}\theta_kg^*}(\frac{\lambda}{\gamma_k}\theta_kB\overline{x}_k + y_k)$ 
\STATE \qquad $x_{k + 1} = x_k - \frac{\gamma_k}{\theta_k}\nabla f(x_k^{\text{md}}) - \frac{\gamma_k}{\theta_k} B^Ty_{k + 1}$ 
\STATE \qquad $x_{k+1}^{\text{ag}} = (1 - \theta_k) x_k^{\text{ag}} + \theta_k x_{k+1}$ 
\STATE \qquad $y_{k+1}^{\text{ag}} = (1 - \theta_k) y_k^{\text{ag}} + \theta_k y_{k+1}$
\STATE \textbf{End For} 
\ENSURE $x_{K}^{ag}$.
\end{algorithmic}
\end{algorithm*}
\end{remark}

\section{Convergence.}\label{sec:convergence}
In this section, we give the convergence results of APDFP. The general convergence condition is given in Theorem \ref{thm1} and specific parameter selection strategies are provided in the Corollary \ref{col1}.
\begin{definition}\label{gap1}
Let $\tilde{z} = (\tilde{x},\tilde{y})$ and $z = (x,y) \in \mathcal{Z} = \mathbb{R}^d \times Y$ and
\begin{equation}\label{pdg1}
Q(\tilde{z},z) = [f(\tilde{x}) + \langle B\tilde{x},y \rangle - g^*(y)] - [f(x) + \langle Bx, \tilde{y}\rangle - g^*(\tilde{y})].
\end{equation}
Define the partial primal dual gap $\mathcal{G}_{B_1 \times B_2}(\tilde{x},\tilde{y})$ \cite{CP} as the follows
\begin{equation}\label{pdg2}
\begin{aligned}
\mathcal{G}_{B_1 \times B_2}(\tilde{x},\tilde{y})
& = max_{z \in B_1 \times B_2} Q(\tilde{z},z) \\  
& = max_{y \in B_2}[f(\tilde{x}) + \langle B\tilde{x},y \rangle - g^*(y)] - min_{x \in B_1}[f(x) + \langle Bx, \tilde{y}\rangle - g^*(\tilde{y})]\\  
\end{aligned}
\end{equation}
where the $B_1 \times B_2 \subset \mathbb{R}^d \times Y$ are bounded set containing the saddle point of \textbf{(\ref{minmax})}.
\end{definition}
\begin{theorem}\label{thm1}
Suppose the function $f$ is $L_f$ smooth convex function and $g$ is convex Lipchitz continuous. Choose the parameter $0 < \lambda \leq \frac{1}{\rho_{\max}(BB^T)}$,  and select a sequence of parameters $\gamma_k$ in Algorithm \ref{APDFP} such that 
\begin{equation}\label{condition}
\left\{
\begin{aligned}
& 0 < \gamma_k \leq \frac{1}{L_f} \\
& \frac{\gamma_{k + 1}}{\gamma_k} \leq \frac{\theta_{k + 1}^2}{\theta_k^2(1 - \theta_{k + 1})} \\
& \frac{\gamma_k}{\gamma_{k + 1}}\leq \frac{1}{(1 - \theta_{k + 1})} \\
& \frac{\gamma_k^2}{\theta_k^2} \mathrm{~is~ increasing~ and ~uniformly ~ bounded}\\
& \theta_1 = 1
\end{aligned}    
\right.
\end{equation}
we then have
\begin{equation}\label{estimate}
\begin{aligned}
& \mathcal{G}_{B_1 \times B_2}(x_{k+1}^{\text{ag}},y_{k+1}^{\text{ag}}) 
\leq \frac{\theta_k^2}{2\gamma_k}\Omega_{1} 
+\frac{\gamma_k}{2}\frac{\rho_{\max}(I - \lambda BB^T)}{\lambda}\Omega_{2}, \\
\end{aligned}
\end{equation}
where $\Omega_{1},\Omega_{2}$ are constant related to diameter of $B_1,B_2$, respectively. 
\end{theorem}

\begin{corollary}\label{col1}
If we choose parameters $\theta_k,\gamma_k$ in Theorem \ref{condition} as
\begin{equation}
\theta_{k} = \frac{2}{k + 1},\gamma_k = \frac{1}{L_f + ck},  0 < c < L_f
\end{equation}
then the condition~\textbf{(\ref{condition})} holds, and we obtain the following estimate:
\begin{equation}\label{estimate2}
\mathcal{G}_{B_1 \times B_2}(x_{k+1}^{\text{ag}},y_{k+1}^{\text{ag}}) 
\leq \frac{2L_f}{(k + 1)^2}\Omega_1 + \frac{2ck}{(k + 1)^2}\Omega_1 
+\frac{1}{2(L_f + ck)}\frac{\rho_{\max}(I - \lambda BB^T)}{\lambda}\Omega_2. \\   
\end{equation}
\end{corollary}

\begin{remark}
\textbf{(Comparison with PDFP)} Compared to PDFP, the APDFP algorithm adopts a decreasing step size $\gamma_k$ and introduces an additional parameter $c$ that, in principle, requires tuning. Nevertheless, in practice, setting $c = 0$, i.e., $\gamma_k := \gamma = \frac{1}{L_f}$, typically yields satisfactory performance. In addition, similar to PDFP, APDFP demonstrates a high degree of insensitivity to the choice of the parameter $\lambda$, which can be safely set to its theoretical upper bound without adversely affecting convergence (see Subsection~\ref{lambda}). These two parameters are primarily determined by the optimization model itself and do not exhibit intricate interdependencies. With the explicit selection of $\theta_k = \frac{2}{k + 1}$, APDFP requires only minimal modifications—namely, the introduction of two aggregate updates, $x_k^{\text{md}}$ and $x_{k+1}^{\text{ag}}$—without imposing any substantial additional burden in parameters tuning. Despite these lightweight changes, APDFP attains both improved theoretical convergence guarantees and enhanced empirical performance as we will show in Section \ref{sec:experiments}.
\end{remark}

\begin{remark}
\textbf{(Limitation to unbounded case)} The convergence analysis in our work relies on the partial primal-dual gap, as the standard primal-dual gap may become unbounded in the absence of domain boundedness assumptions. In contrast, the APD \cite{APD} addresses unbounded scenarios using a perturbation-based termination criterion developed by Monteiro and Svaiter \cite{monteiro2010complexity, monteiro2011complexity, monteiro2013iteration}. We attempted to establish similar convergence results for APDFP. However, we encountered theoretical obstacles that we were unable to resolve. The core difficulty lies in the reduced degree of freedom within the APDFP algorithm. Specifically, while APD allows independent algorithmic parameters $\gamma_k$ and $\tau_k$ of the primal and dual updates, APDFP ties fixed relationship (i.e., $\gamma_k$ for the primal variable and $\frac{\lambda}{\gamma_k}$ for the dual). This coupling restricts the flexibility needed to apply the perturbation-based analysis framework. As such, we extend the convergence theory of APDFP to the unbounded domain setting to the future work.
\end{remark}

\begin{remark}
\textbf{(Extension to relative smooth)} It is observed that the primal updates in APD, APDFP, and the first subproblem of AADMM can all be interpreted as variants of gradient descent. For APDFP, the convergence analysis requires the objective function \( f \) to be uniformly smooth—that is, its gradient must have a uniformly Lipschitz constant. 
However, for problems where \( f \) does not satisfy uniform Lipschitz continuity but is instead relatively smooth with respect to some reference function \( h \) (i.e., \( h(x) - f(x) \) is convex, the conventional smoothness setting corresponds to \( h(x) = \frac{L_f}{2}\|x\|_2^2 \)), both APD and AADMM remain applicable by substituting standard accelerated gradient descent with accelerated Bregman gradient descent \cite{BPG}.
 Unfortunately, APDFP is currently not equipped to handle such problems, due to limitations inherent in its update rule structure.

\end{remark}

\section{Numerical Experiments}\label{sec:experiments}
In this section, we evaluate the performance of the proposed APDFP through experiments on graph-guided logistic regression and 2D CT reconstruction. As outlined in the introduction, we compare APDFP against aforementioned algorithms that incorporate Nesterov-type acceleration: APD (Algorithm \ref{APD}) and AADMM (Algorithm \ref{AADMM}). Additionally, we include  IPDFP (Algorithm \ref{iPDFP}) in our comparison. To highlight the benefit of acceleration schemes, we also benchmark against their non-accelerated counterparts: PDFP (Algorithm \ref{PDFP}), the LPDHGm (Algorithm \ref{LPDHGm}) and the LP-ADMM (Algorithm \ref{LPADMM}).
All algorithms are implemented in the Python programming language and evaluated on a desktop equipped with an Intel\textsuperscript{\textregistered} Core\textsuperscript{\texttrademark}  i9 processor. The iterations were terminated once the relative error between consecutive iterates, i.e., $\frac{\|x_{k+1} - x_k\|}{\|x_k\|}$
fell below the prescribed threshold of $10^{-3}$, or when the maximum number of iterations was reached.

\subsection{Graph-Guided Logistic regression}
In this subsection, we present numerical experiments on graph-guided logistic regression to evaluate the performance of the proposed APDFP algorithm. The mathematical formulation of the graph-guided logistic regression model is given as follows:
\begin{equation}\label{GGLST}
\underset{x \in \mathbb{R}^d}{\min}~\frac{1}{N} \sum_{i = 1}^{N} \log \big(1 + \exp(-b_is_i^Tx) \big) + \frac{\mu_1}{2}\lVert x \rVert_2^2 + \mu_2\lVert Bx \rVert_1, 
\end{equation}
where \( b_i \in \{-1, 1\} \) denote the label corresponding to the sample \( s_i \in \mathbb{R}^d \), and \( N \) is the total number of samples. The matrix \( B \) represents the precision matrix derived via sparse inverse covariance estimation~\cite{GLasso}. The regularization parameters \( \mu_1 \) and \( \mu_2 \) control the balance between the data fidelity term and the regularization terms. We evaluate the proposed methods on three real-world datasets obtained from the LIBSVM repository~\cite{LIBSVM}, namely \textit{a9a}, \textit{mushrooms}, and \textit{w8a}. Detailed dataset configurations are summarized in Table~\ref{datasets}. The regularization parameters \( \mu_1 \) and \( \mu_2 \) are selected via grid search to maximize testing accuracy. For all baseline algorithms, we follow the parameter tuning protocols described in~\cite{APD,AADMM}, manually adjusting parameters to ensure optimal empirical performance. For APDFP, we set $\gamma_k := \frac{1}{L_f}$ and $\lambda = \frac{1}{\rho_{\max}(BB^T)}$. To benchmark convergence behavior, we first run the PDFP algorithm for 10{,}000 iterations to approximate the optimal function value. Subsequently, we assess algorithmic performance by plotting the relative error of the training loss and the testing accuracy (log scale) across iterations. As reported in Table~\ref{datasets}~((a), (c), and (e)), the non-accelerated methods exhibit similar convergence behavior. In contrast, the accelerated variants---APD, APDFP, and AADMM---converge substantially faster. In particular, APDFP and APD exhibits the faster convergence compared to AADMM on the three datasets. IPDFP exhibits faster convergence on \textit{a9a} and \textit{mushroom} compared to non-accelerated algorithms. This advantage, however, vanishes on \textit{w8a}. Testing accuracy results, as shown in Table~\ref{datasets}~((b), (d), and (f)), further confirm that the accelerated methods reach higher classification accuracy with fewer iterations.

\begin{table}[htp]
\setlength{\belowcaptionskip}{0.2cm} 
\caption{Configuration of different data sets.}
\centering
\begin{tabular}{|c|c|c|c|c|c|c|c|c|c|c|c|}
\hline
\multicolumn{2}{|c|}{Data sets}
& \multicolumn{2}{|c|}{$\sharp$ of samples}
& \multicolumn{2}{|c|}{$\sharp$ of train}
& \multicolumn{2}{|c|}{$\sharp$ of test}
& \multicolumn{2}{|c|}{$\sharp$ of features} \\
\hline
\multicolumn{2}{|c|}{a9a}
& \multicolumn{2}{|c|}{$32,561$}
& \multicolumn{2}{|c|}{$26,053$}
& \multicolumn{2}{|c|}{$6,508$}
& \multicolumn{2}{|c|}{$123$} \\
\hline
\multicolumn{2}{|c|}{mushrooms}
& \multicolumn{2}{|c|}{$8,124$}
& \multicolumn{2}{|c|}{$6,451$}
& \multicolumn{2}{|c|}{$1,673$}
& \multicolumn{2}{|c|}{$113$} \\
\hline
\multicolumn{2}{|c|}{w8a}
& \multicolumn{2}{|c|}{$49,749$}
& \multicolumn{2}{|c|}{$39,732$}
& \multicolumn{2}{|c|}{$10,017$}
& \multicolumn{2}{|c|}{$301$} \\
\hline
\end{tabular}
\label{datasets}
\end{table}

\begin{figure}[htbp!]
\centering
\subfloat[a9a (train)]{
\includegraphics[width = 2.5 in]{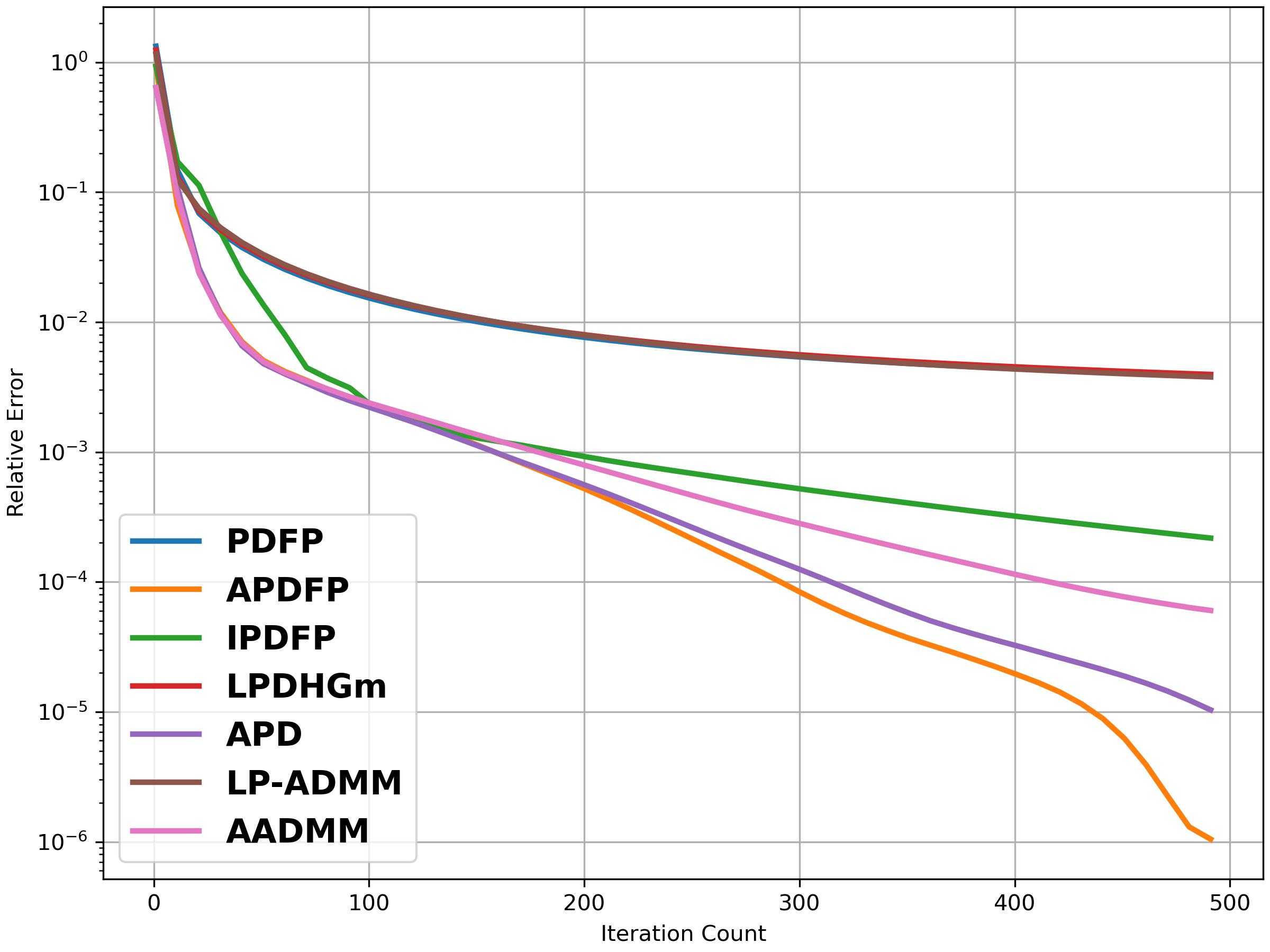}
}
\subfloat[a9a (accuracy) ]{
\includegraphics[width = 2.5 in]{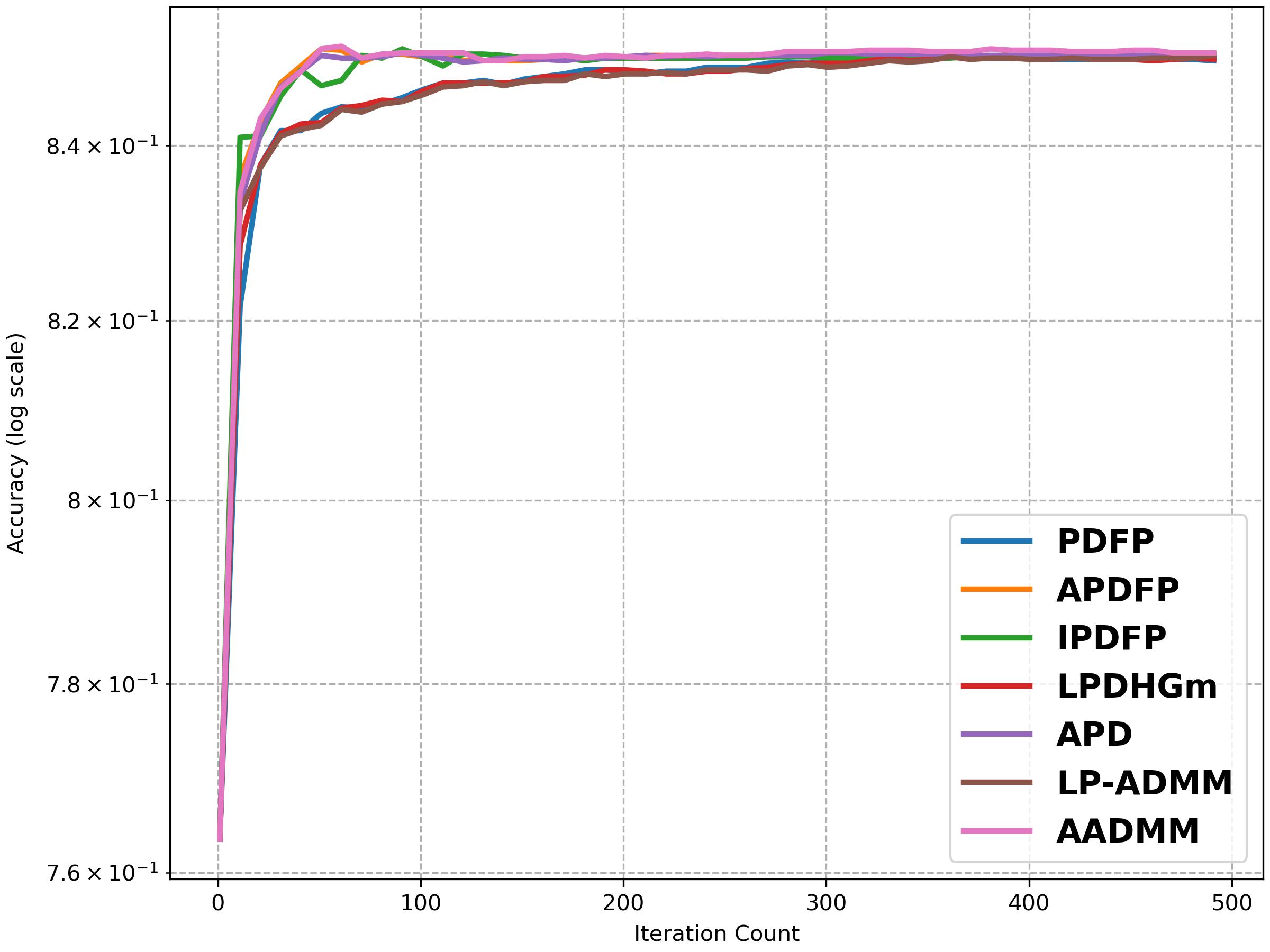}
}\\[-1mm]
\subfloat[mushrooms (train)]{
\includegraphics[width = 2.5 in ]{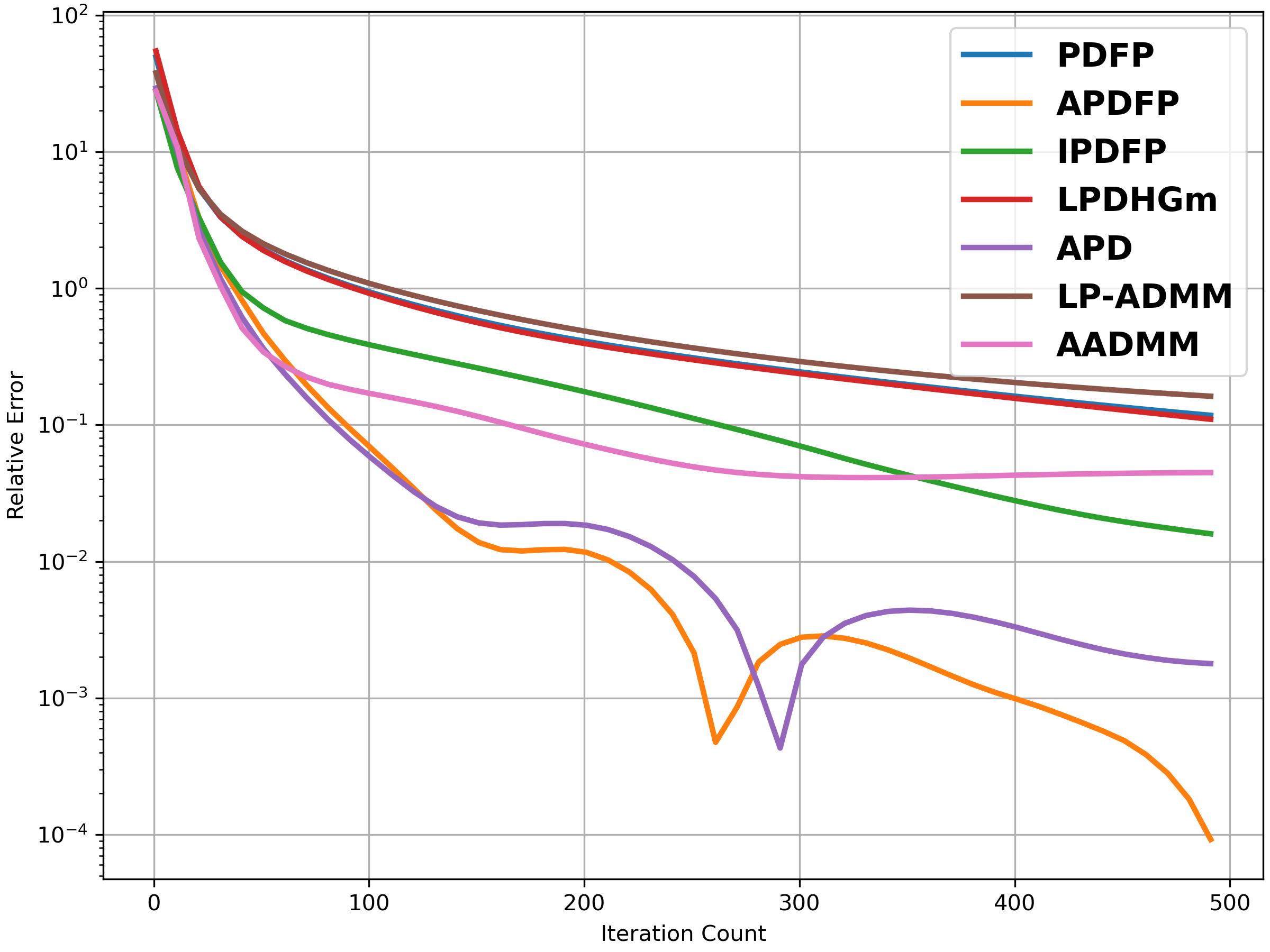}
}
\subfloat[mushrooms (accuracy) ]{
\centering
\includegraphics[width = 2.5 in]{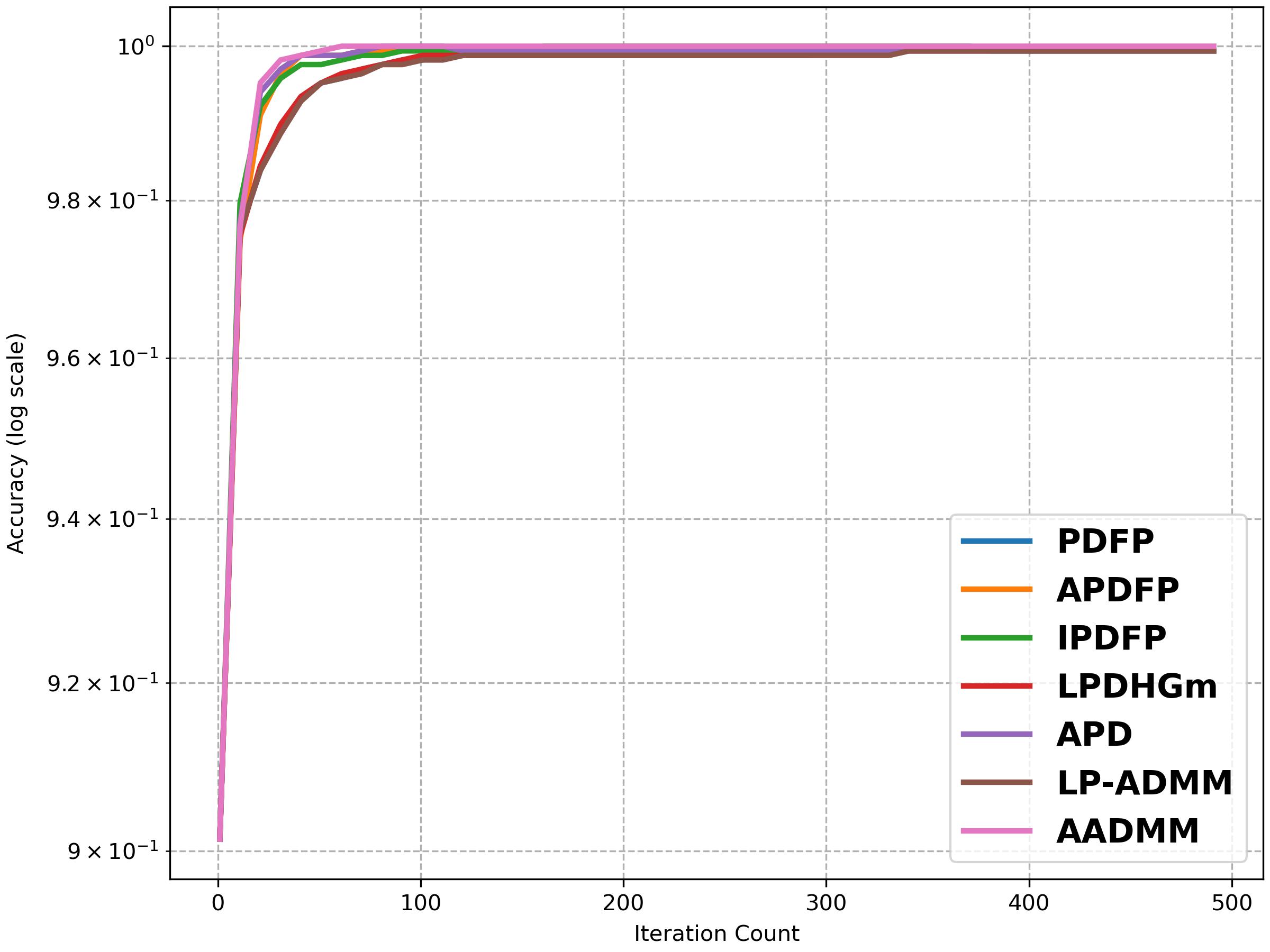}
}\\[-1mm]
\subfloat[w8a (train)]{
\includegraphics[width = 2.5 in ]{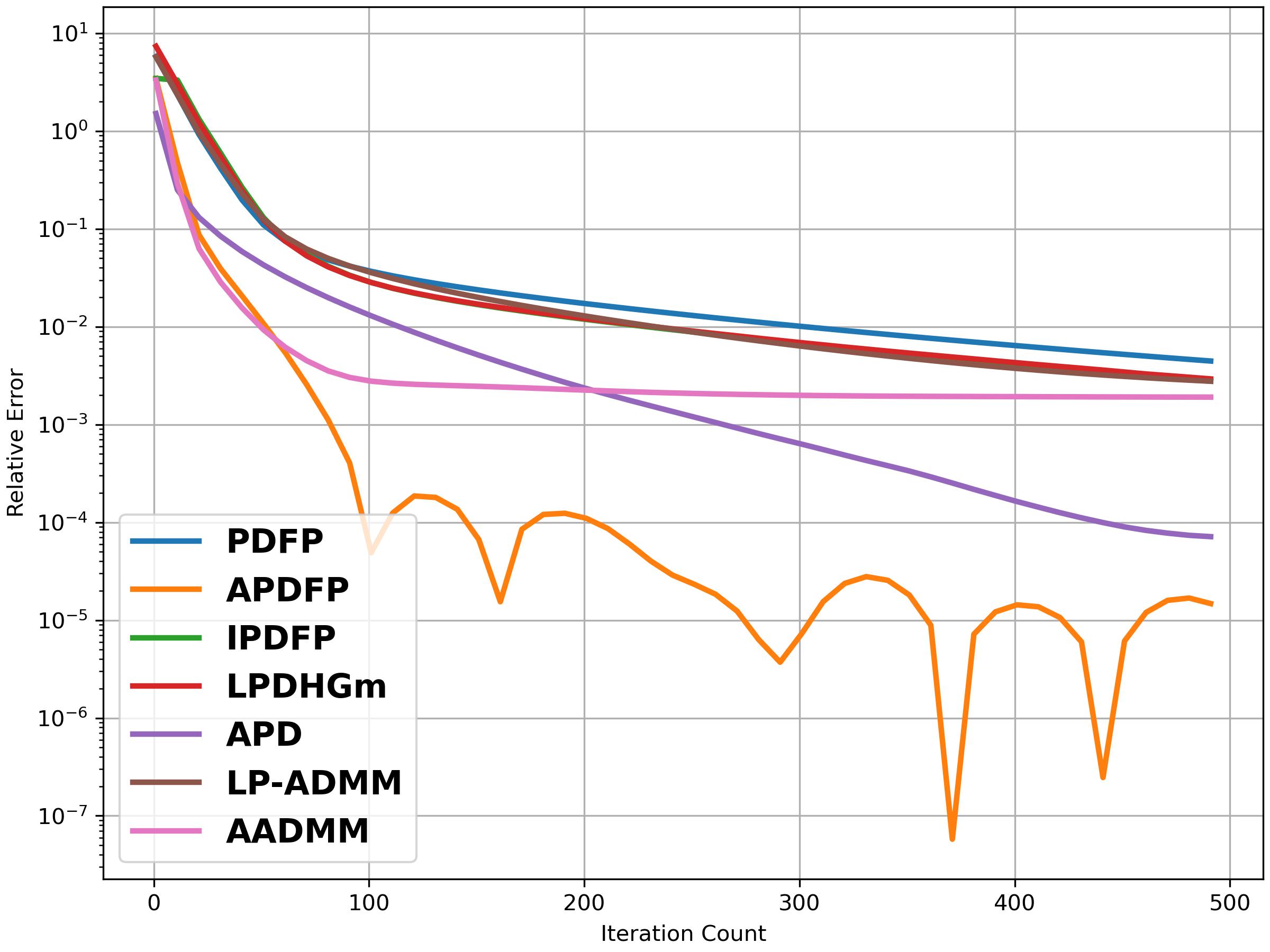}
}
\subfloat[w8a (accuracy) ]{
\includegraphics[width = 2.5 in]{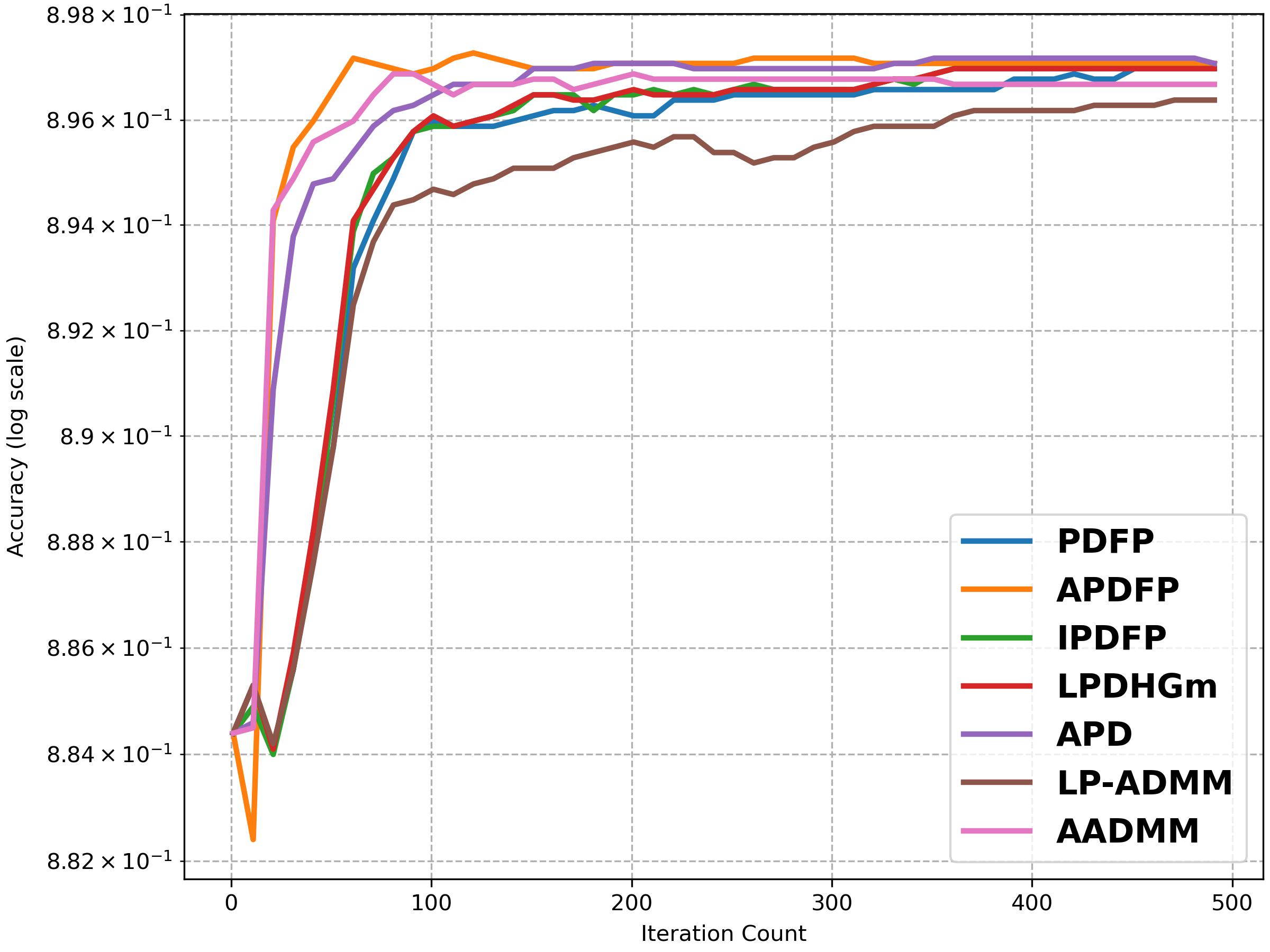}
}
\caption{The relative error of training and testing accuracy on different data sets.}
\label{GGLST1}
\end{figure} 

\subsection{2D CT reconstruction.}
2D CT Reconstruction refers to the process of reconstructing an image of an object from its 2D X-ray projections taken at various angles. The optimization problem using TV-$L_2$ model takes the following form
\begin{equation}\label{TVl2}
 \min_{x \in \mathbb{R}^d}~\frac{1}{2}\lVert \mathcal{A}x - y \rVert_2^2 + \mu \lVert \nabla x \rVert_{1,2}, 
\end{equation}
where 
\begin{itemize}
  \item \( x \in \mathbb{R}^d \) denote the vectorized image to be reconstructed and \( d = 512 \times 512 \) corresponds to the total number of pixels in the 2D image.
  \item The discrete X-ray transform \( \mathcal{A} \) \cite{Siddon,Xray} maps the image \( x \in \mathbb{R}^d \) to its projection measurements by computing the intersection lengths between each ray path and the image pixels, i.e.,
  \begin{equation*}
  \begin{aligned}
  & y_{a,v} = \sum_{(i,j) \in D_{a,v}} \ell_{(a,v),(i,j)}\tilde{x}_{i,j}, a \leq n_d, v \leq n_v; \\
  & D_{a,v} = \{(i,j):\ell_{(a,v),(i,j)} > 0,i \leq 512,j \leq 512\}
  \end{aligned}
  \end{equation*}
Here, $\tilde{x}$ denotes the image prior to concatenation, and $i$ and $j$ index the pixels. 
The symbols $a$ and $v$ correspond to the indices of the $a$-th projection angle and the $v$-th ray, respectively. 
$D_{a,v}$ denotes the set of pixel indices that have a nontrivial intersection with the $v$-th beam at the $a$-th view. 
In our setting, The size of the detector line is \( n_v = 512 \), and the number of projection angles is \( n_a = 360 \). To alleviate the computational cost associated with evaluating the gradient of the data fidelity term, we adopt the parallelization strategy proposed in~\cite{Xray} for both the X-ray transform and its adjoint.
  \item  $y$ is the noisy  projection vector which is obtained by adding a Gaussian noise $\varepsilon$ with zero mean and variance $0.03$ to the projection data of ground truth image $x_0$, i.e., $b = \mathcal{A}x_0 + \varepsilon$.
  \item $\mu = 10^{-3}$ is the regularization parameter.
  \item $\nabla $ is the discrete gradient operator that maps a two dimension image $x$ to a vector field consisting of vertical and horizontal finite difference as follows
  \begin{equation*}
    (\nabla x)_{i,j} =\left (
    \begin{aligned}
    (\nabla x)_{i,j}^1, 
    (\nabla x)_{i,j}^2 
    \end{aligned}
    \right )^T
  \end{equation*}
  where 
   \begin{equation*}
   \begin{aligned}
    (\nabla x)_{i,j}^1 &= \left \{
    \begin{aligned}
    & x_{i +1,j} - x_{i,j}, ~\mathrm{if} ~ 1 \leq i< 512 \\
    & 0, ~\mathrm{if} ~ i= 512
    \end{aligned}
    \right. \\
    (\nabla x)_{i,j}^2 &= \left \{
    \begin{aligned}
    & x_{i,j + 1} - x_{i,j}, ~\mathrm{if} ~ 1 \leq j< 512 \\
    & 0, ~\mathrm{if} ~ j= 512
    \end{aligned}
    \right. 
    \end{aligned}
  \end{equation*} 
The isotropic total variation  $\lVert \nabla x \rVert_{1,2}$ is  defined by 
\begin{equation*}
\begin{aligned}
\hspace*{-6mm}\lVert \nabla x \rVert_{1,2}
& = \sum_{i = 1}^{512}\sum_{j = 1}^{512} \lVert (\nabla x)_{i,j} \rVert_2 
 = \sum_{i = 1}^{512}\sum_{j = 1}^{512} \sqrt{((\nabla x)_{i,j}^1)^2 + ((\nabla x)_{i,j}^2)^2}.
\end{aligned}
\end{equation*}
\end{itemize}   
For all baseline algorithms, we follow the parameter-tuning protocols outlined in~\cite{APD,AADMM} and further adjust the parameters manually to obtain the best empirical performance. 
For APDFP, we set $\gamma_k := \tfrac{1}{L_f}$, where $L_f = \rho_{\max}(\mathcal{A}^\top \mathcal{A})$ is estimated using the power method \cite{golub2013matrix}, and choose $\lambda = \tfrac{1}{8}$, since the maximum eigenvalue of $\nabla^\top \nabla$ is $\rho_{\max}(\nabla^\top \nabla) = 8$~\cite{CP}.
 The function values and PSNRs as a function of iteration number are shown in Figure~\ref{CTobj}. From the objective value plot, we observe that the algorithms without Nesterov acceleration exhibit similar convergence behavior. In contrast, the accelerated variants converge substantially faster and it is interesting to find that the three accelerated algorithm exhibit nearly the same convergence speed. A practical advantage of APDFP is that it requires no parameter tuning.  From PSNR plot, it is evident that the accelerated algorithms achieve higher PSNR values at earlier iterations compared to their non-accelerated counterparts. However, PSNR typically decreases after reaching a peak as the number of iterations increases. Although the objective values of the accelerated algorithms eventually converge to similar levels (The non-accelerated algorithms require significantly more iterations to reach the optimal function value which exceed the maximum iteration), the difference in PSNR behavior suggests that they may converge to distinct optimal point. Therefore, for the CT reconstruction task considered, early stopping is generally recommended for accelerated algorithms to preserve image quality. The reconstructed images in Figure~\ref{CT2} further demonstrate the effectiveness of the proposed algorithms.

\begin{figure}[htbp!]
\centering
\subfloat[objective value ]{
\centering
\includegraphics[width = 2.5 in]{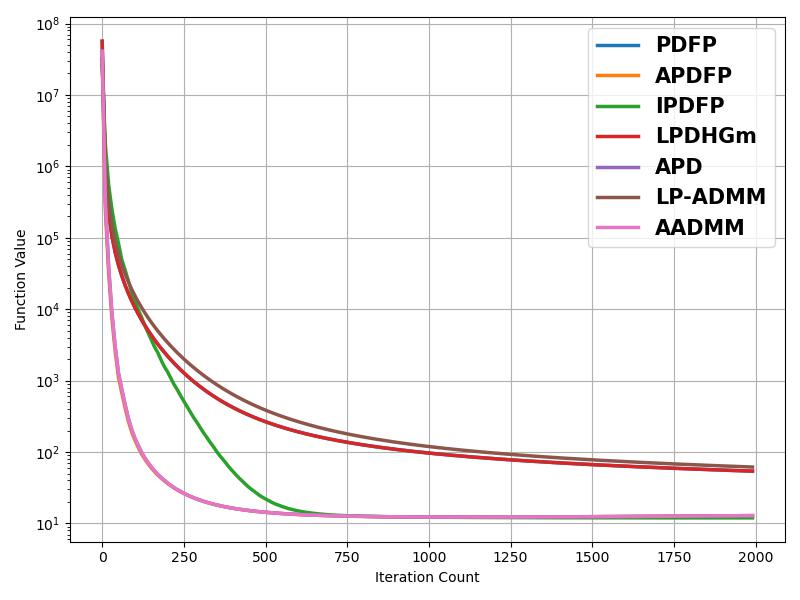}
}
\subfloat[PSNR ]{
\centering
\includegraphics[width = 2.5 in]{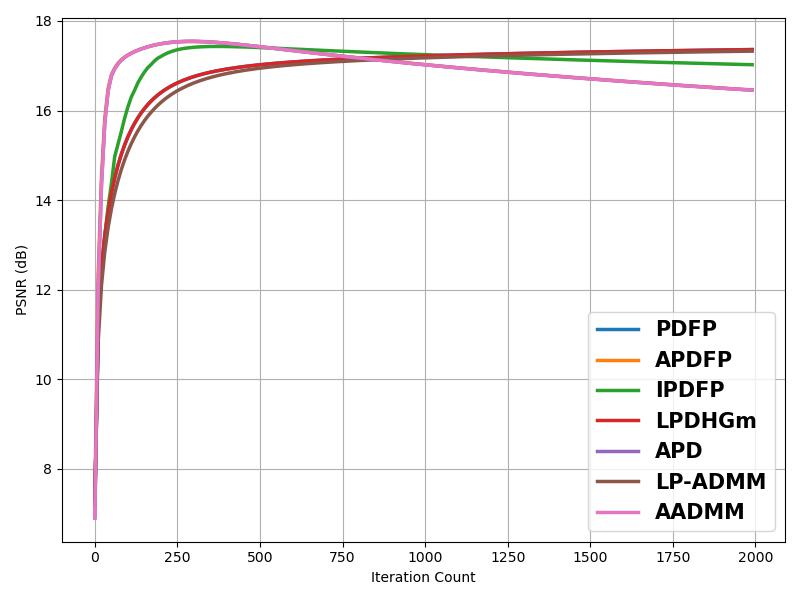}
}
\caption{The function value and PSNR versus the iteration number.}
\label{CTobj}
\end{figure}

\begin{figure}[!h]
\centering
\renewcommand{\arraystretch}{0.0} % 行间距压缩

\begin{tabular}{@{}c@{\hskip 8pt}c@{}} % 无左右边距，控制水平间距
% --- 第一行 ---
\includegraphics[trim=60 50 60 50, clip, width=0.28\linewidth]{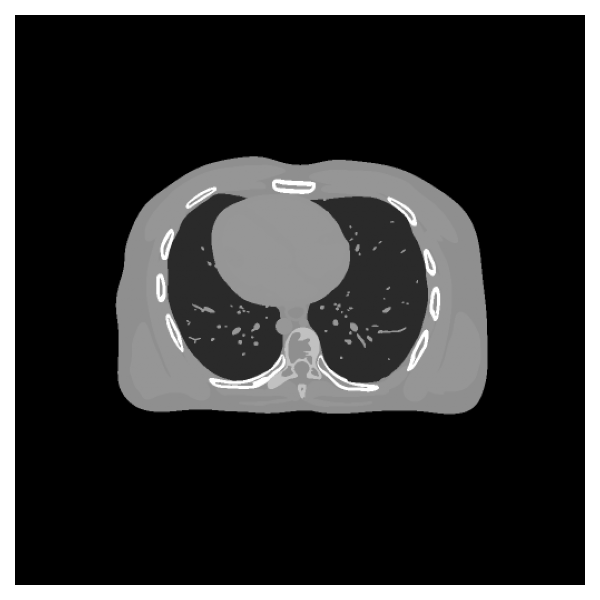} &
\includegraphics[trim=60 50 60 50, clip, width=0.28\linewidth]{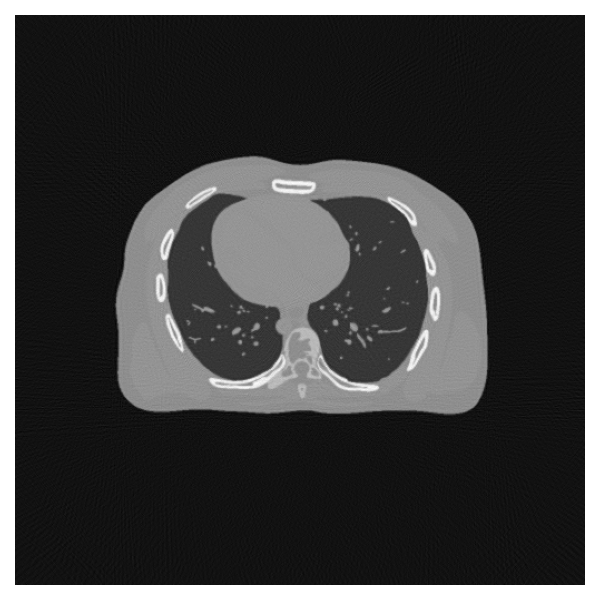} \\[2mm]
\scriptsize Ground Truth &
\scriptsize PDFP, $\mathrm{PSNR} = 17.36$
\\[2mm]

% --- 第二行 ---
\includegraphics[trim=60 50 60 50, clip, width=0.28\linewidth]{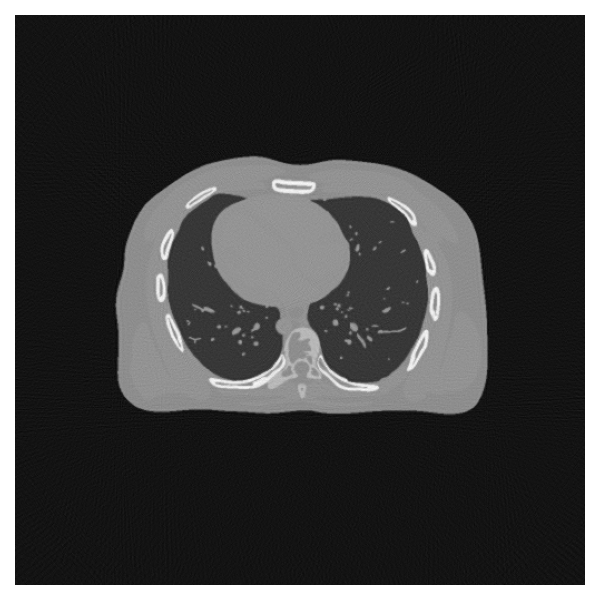} &
\includegraphics[trim=60 50 60 50, clip, width=0.28\linewidth]{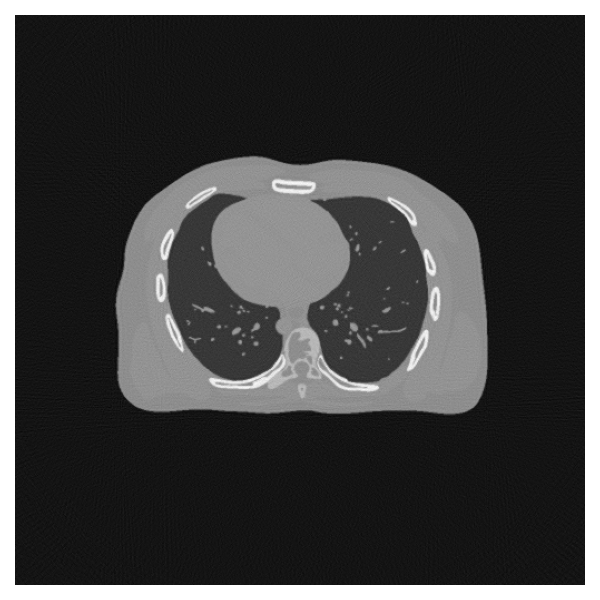} \\[2mm]
\scriptsize APDFP, $\mathrm{PSNR} = 17.54$ &
\scriptsize IPDFP, $\mathrm{PSNR} = 17.43$
\\[2mm]

% --- 第三行 ---
\includegraphics[trim=60 50 60 50, clip, width=0.28\linewidth]{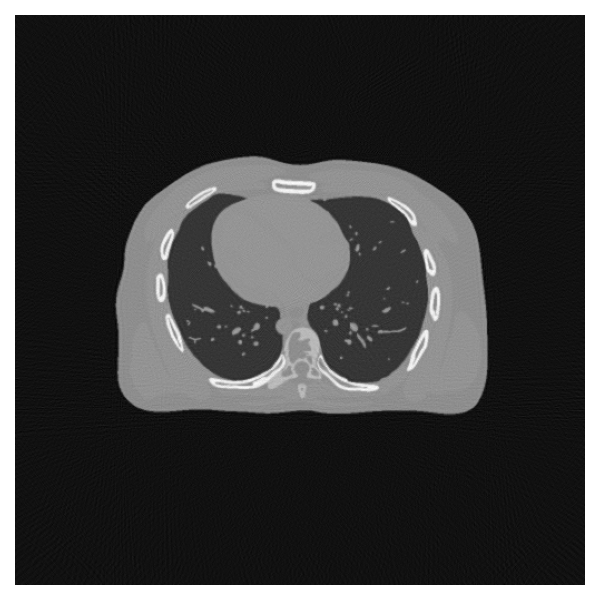} &
\includegraphics[trim=60 50 60 50, clip, width=0.28\linewidth]{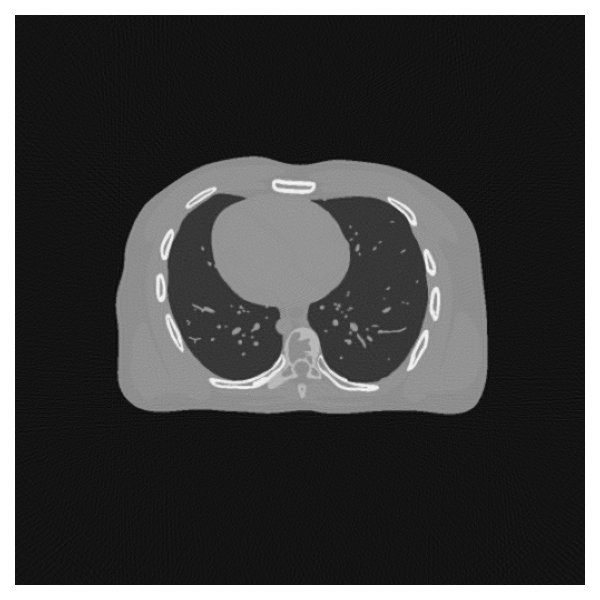} \\[2mm]
\scriptsize LPDHG, $\mathrm{PSNR} = 17.36$ &
\scriptsize APD, $\mathrm{PSNR} = 17.54$
\\[2mm]

% --- 第四行 ---
\includegraphics[trim=60 50 60 50, clip, width=0.28\linewidth]{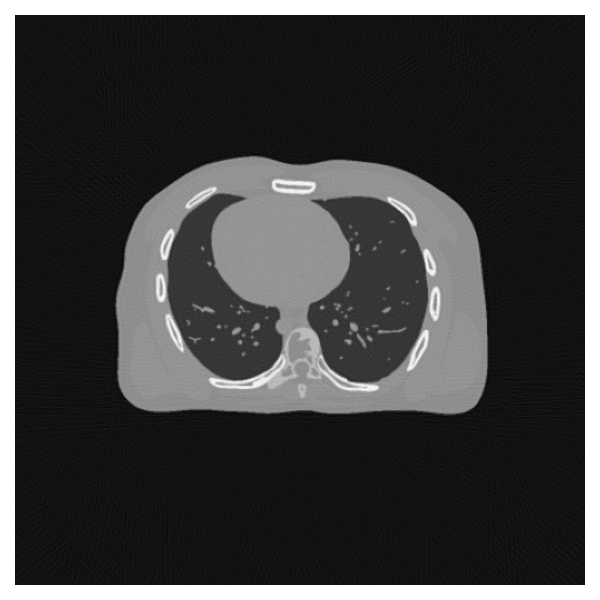} &
\includegraphics[trim=60 50 60 50, clip, width=0.28\linewidth]{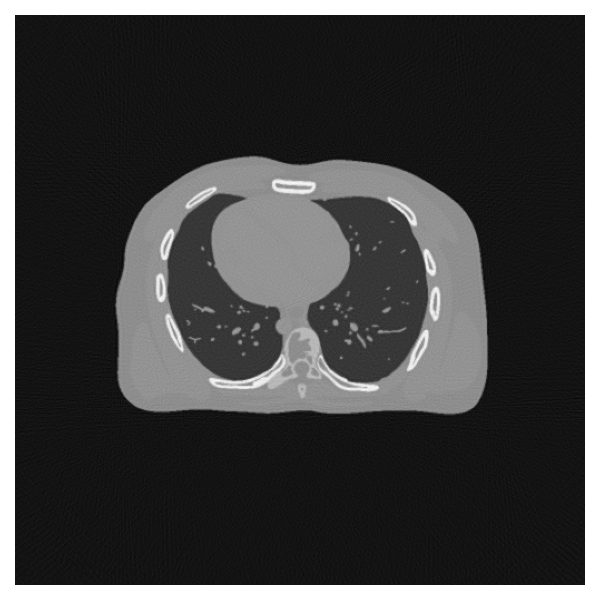} \\[2mm]
\scriptsize LADMM, $\mathrm{PSNR} = 17.39$ &
\scriptsize AADMM, $\mathrm{PSNR} = 17.54$
\end{tabular}
\vspace{2mm}
\caption{One slice of the reconstructed image produced by different algorithms. The corresponding PSNR values are indicated below each image.}
\label{CT2}
\end{figure}

\subsection{Effect of $\lambda$.}\label{lambda}
One notable property of the PDFP algorithm is its insensitivity to the algorithmic parameter $\lambda$. Empirically, setting $\lambda = \frac{1}{\rho(BB^T)}$ typically yields satisfactory performance without compromising convergence behavior~\cite{PDFP}. In this section, we investigate whether this robustness with respect to $\lambda$ also extends to APDFP. To this end, we evaluate the performance of APDFP using a range of $\lambda$ values: $\frac{1}{\rho(BB^T)}$, $\frac{0.7}{\rho(BB^T)}$, $\frac{0.5}{\rho(BB^T)}$, $\frac{0.3}{\rho(BB^T)}$, and $\frac{0.1}{\rho(BB^T)}$, across the graph-guided logistic regression and CT reconstruction. As illustrated in Figures~\ref{GGLST-lam} and~\ref{CT-2D-lam}, APDFP demonstrates nearly identical convergence profiles across all tested values of $\lambda$ in both problems, thereby confirming its robustness with respect to this parameter.

\begin{figure}[htbp!]
\centering
\subfloat[a9a (train)]{
\includegraphics[width = 2.5 in ]{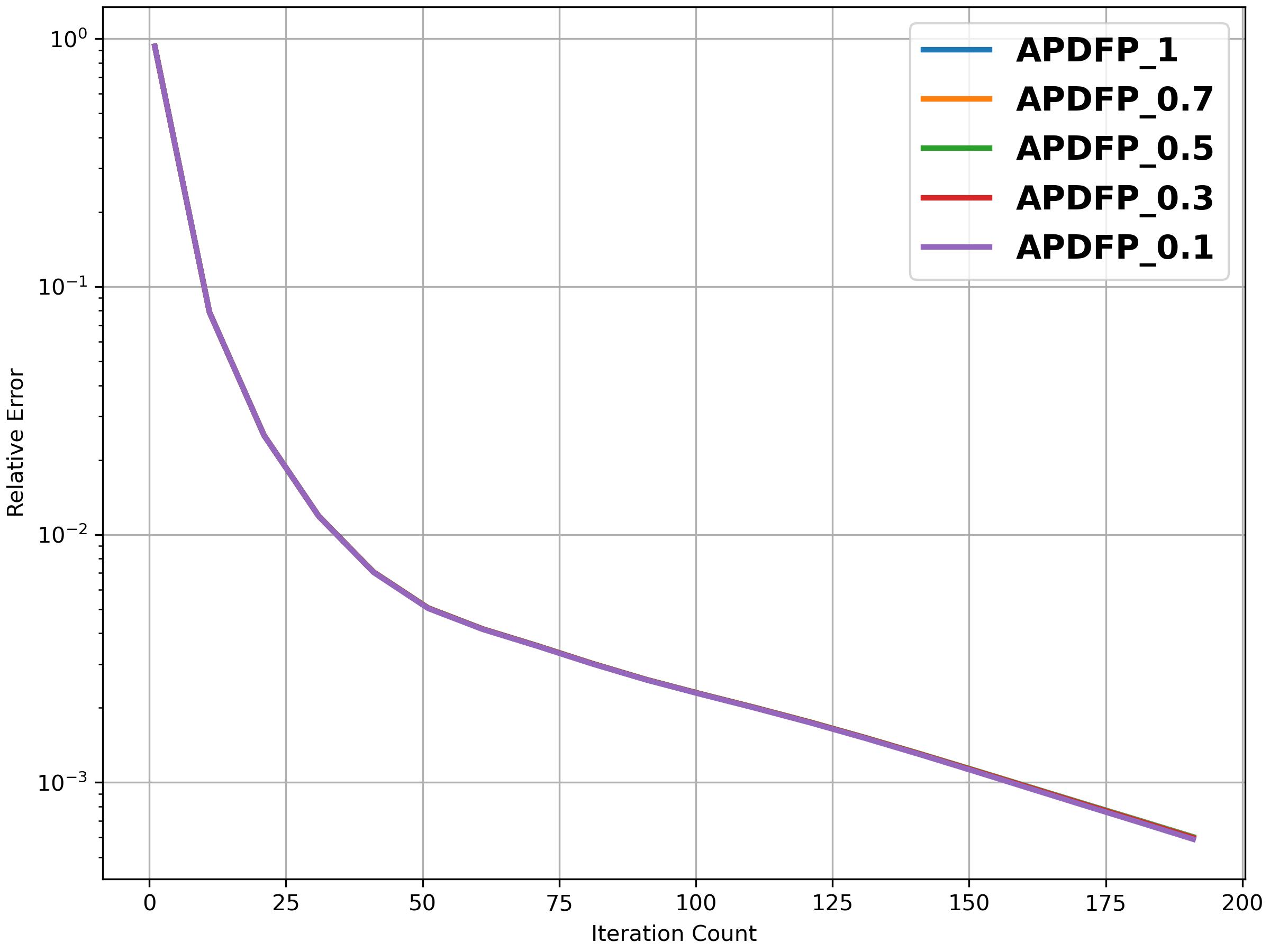}
}
\subfloat[a9a (accuracy) ]{
\includegraphics[width = 2.5 in]{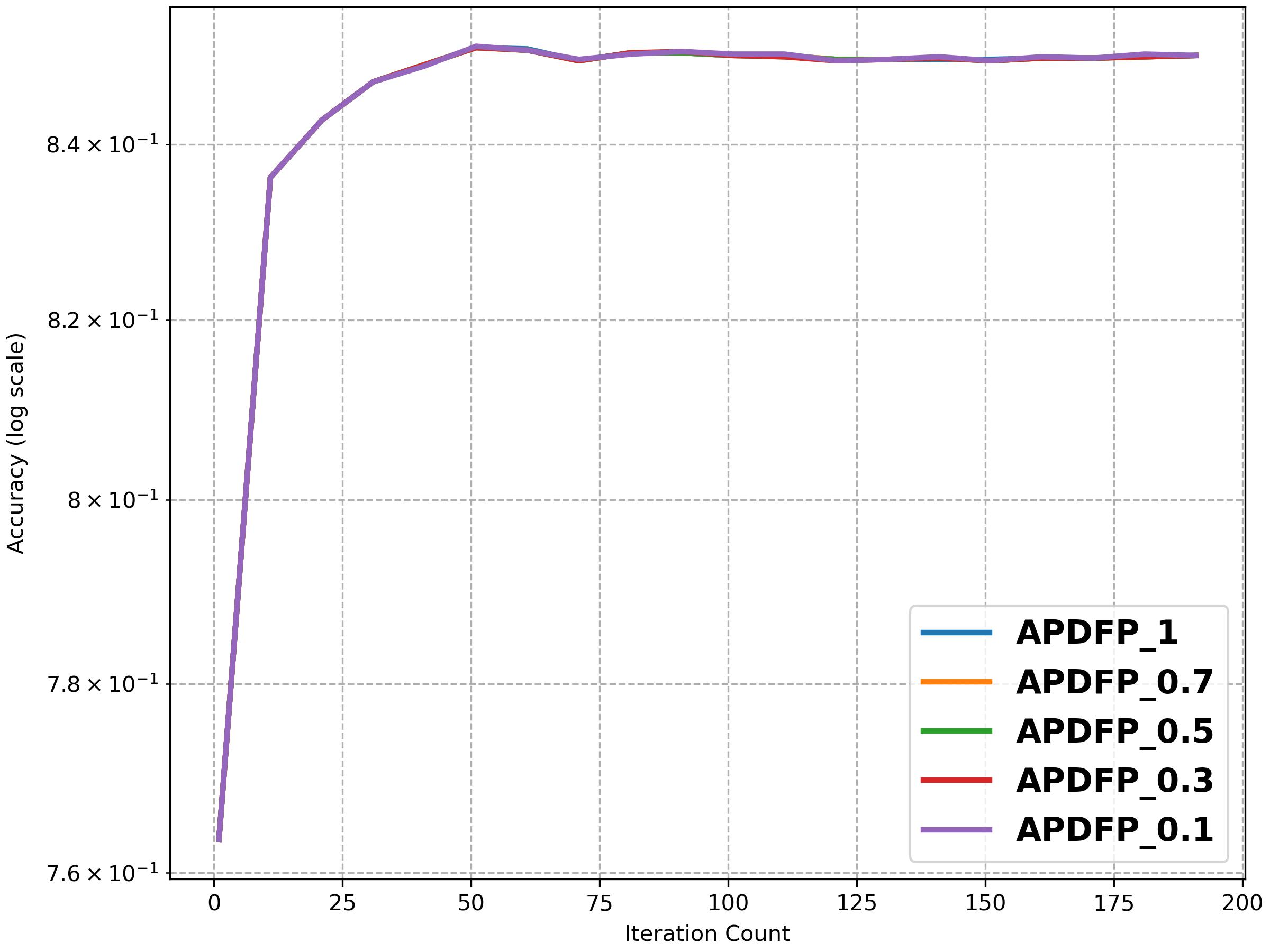}
}\\[-1mm]
\subfloat[mushrooms (train)]{
\includegraphics[width = 2.5 in ]{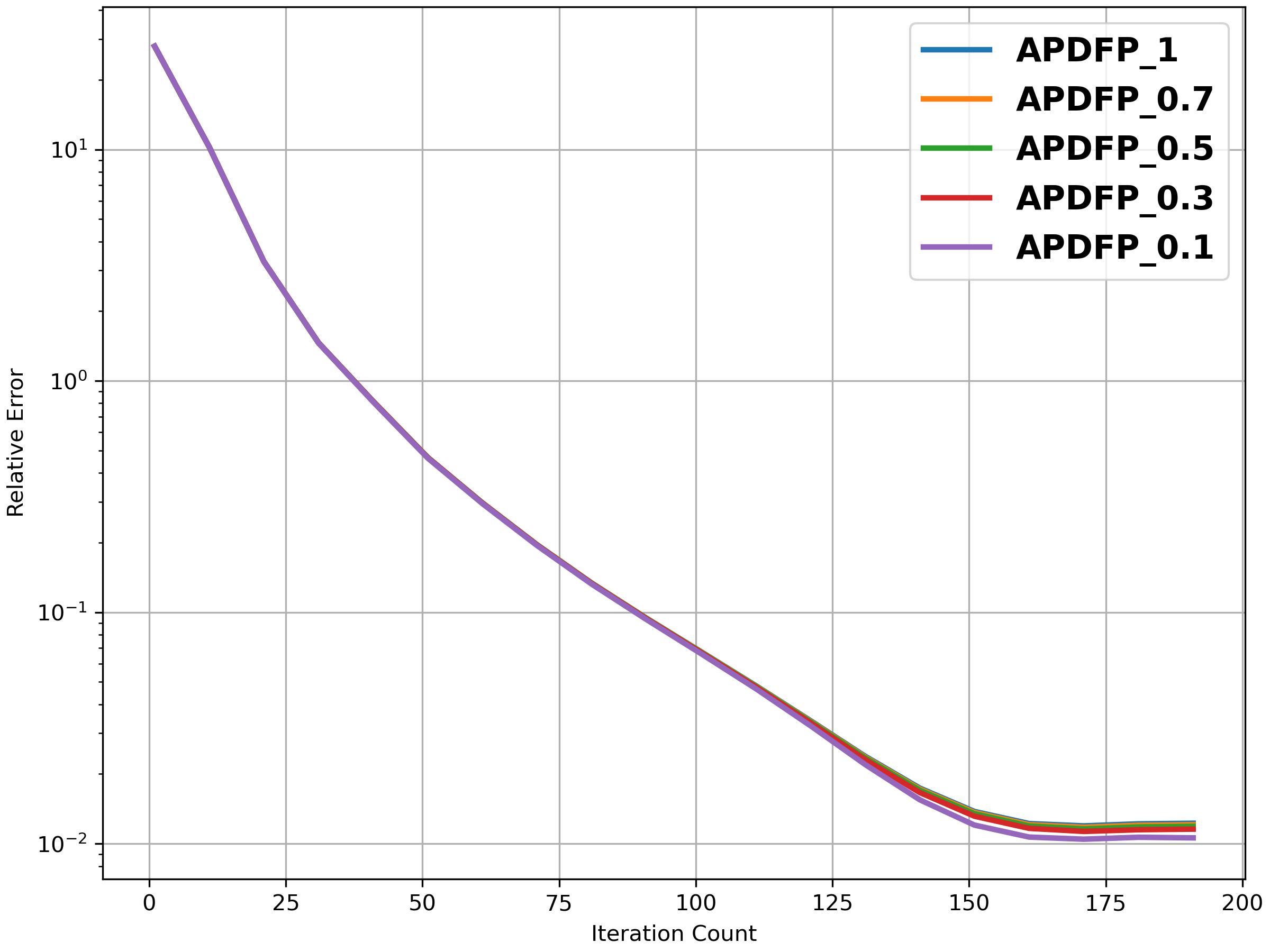}
}
\subfloat[mushrooms (accuracy) ]{
\includegraphics[width = 2.5 in]{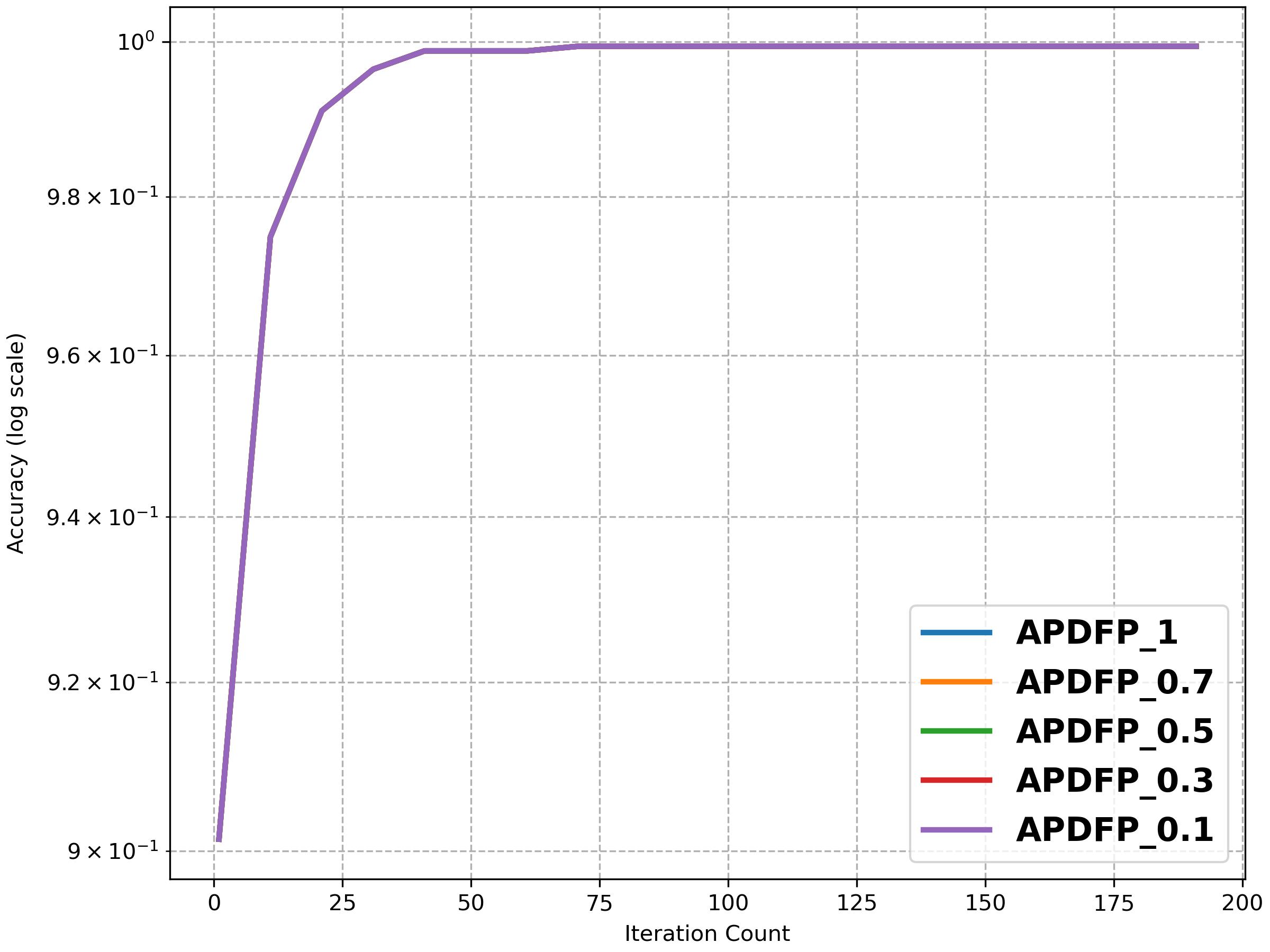}
}\\[-1mm]
\subfloat[w8a (train)]{
\includegraphics[width = 2.5 in ]{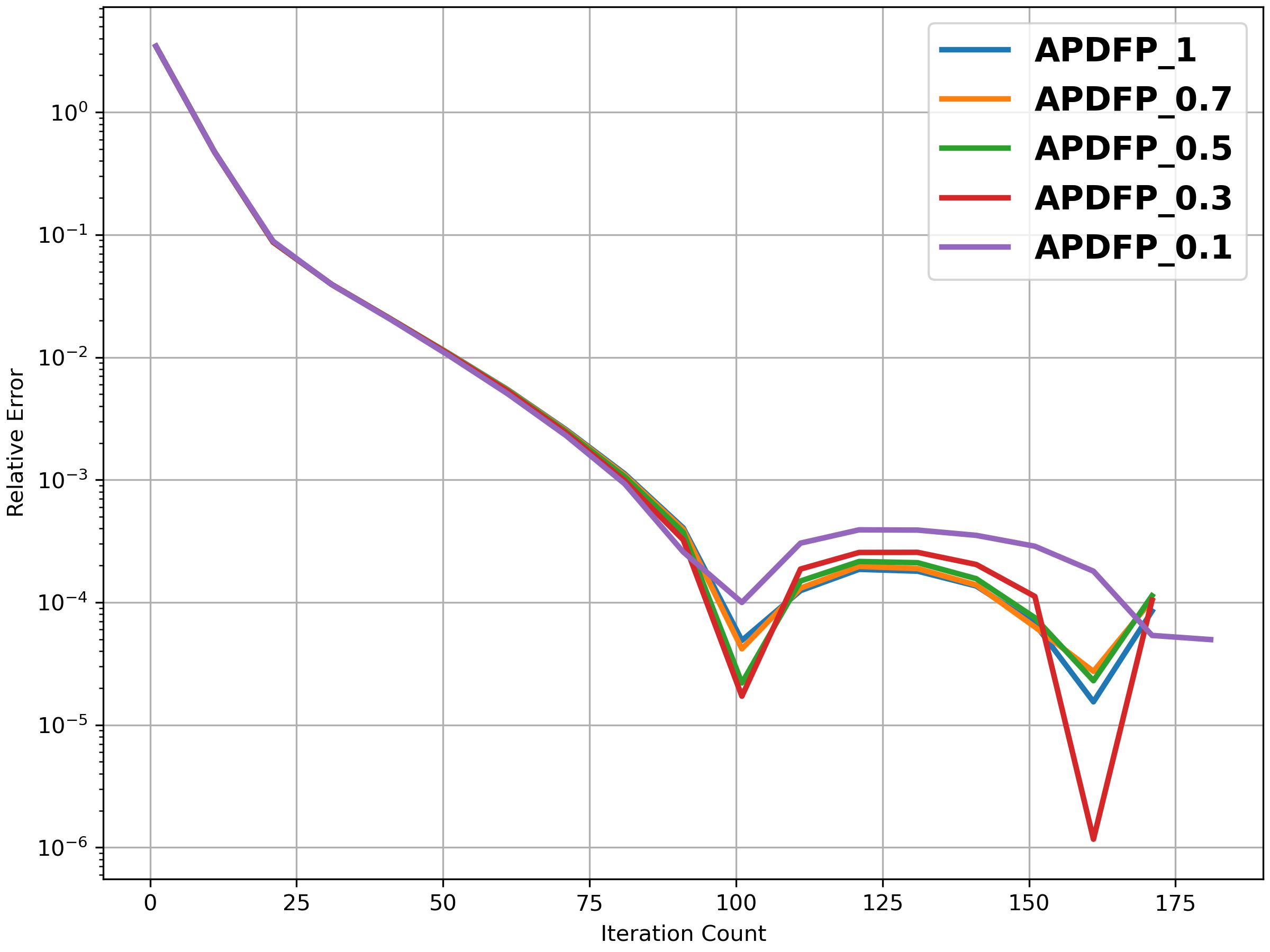}
}
\subfloat[w8a (accuracy) ]{
\includegraphics[width = 2.5 in]{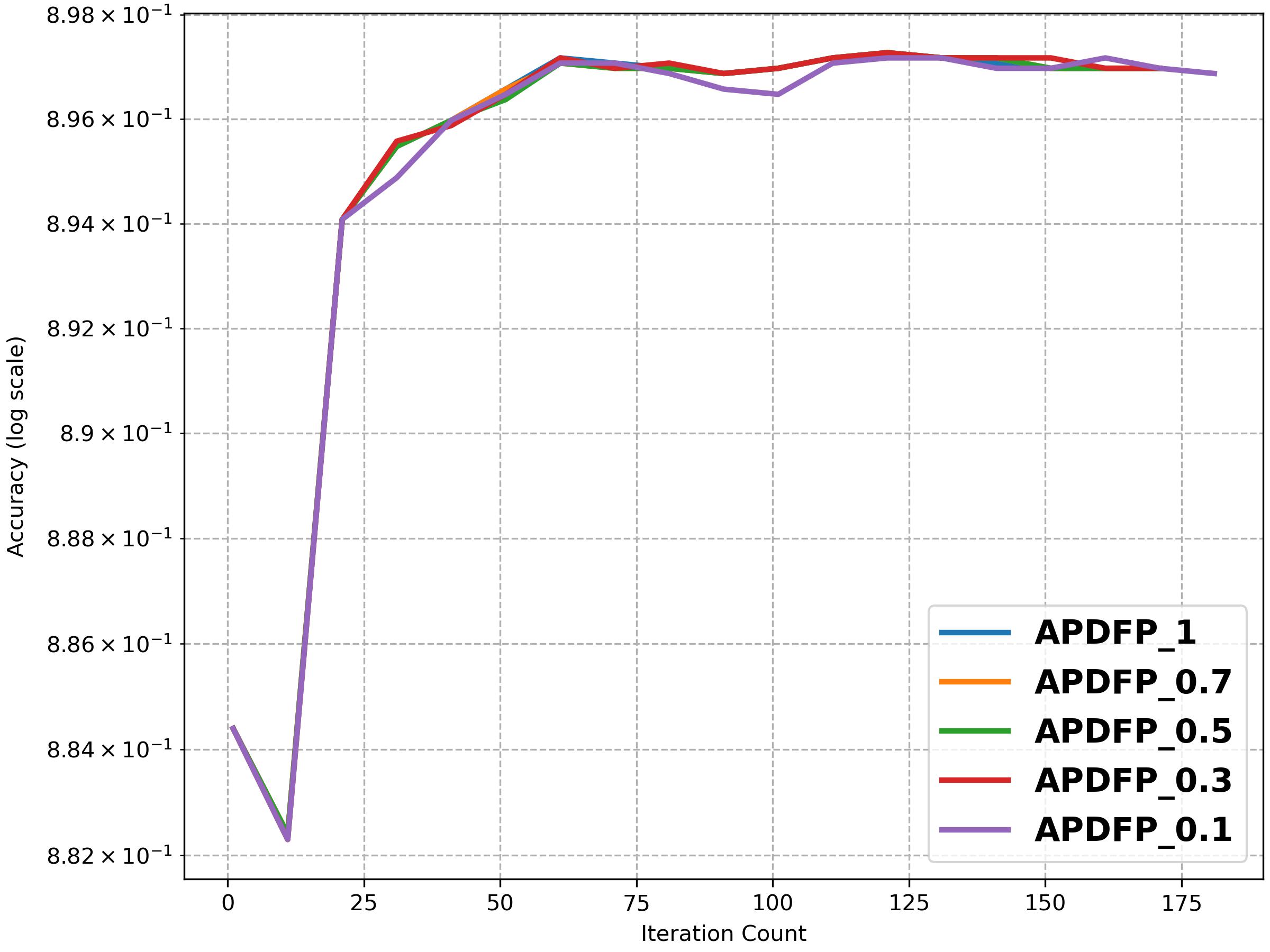}
}
\caption{The relative error of the objective function value and the testing accuracy with respect to the number of iterations for different values of $\lambda$ across various datasets.}
\label{GGLST-lam}
\end{figure}

\begin{figure}[htbp!]
\centering
\subfloat[function value]{
    \includegraphics[width = 2.5in]{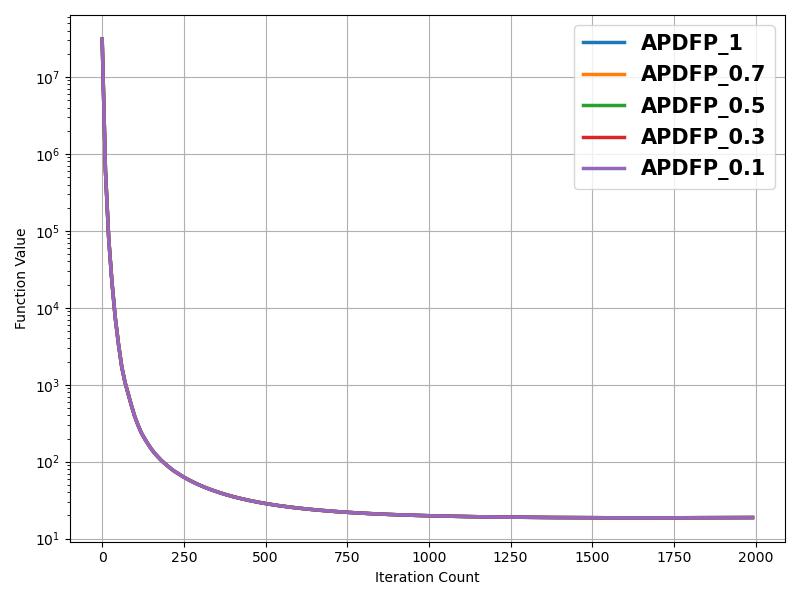}
}
\subfloat[PSNR]{
    \includegraphics[width = 2.5in]{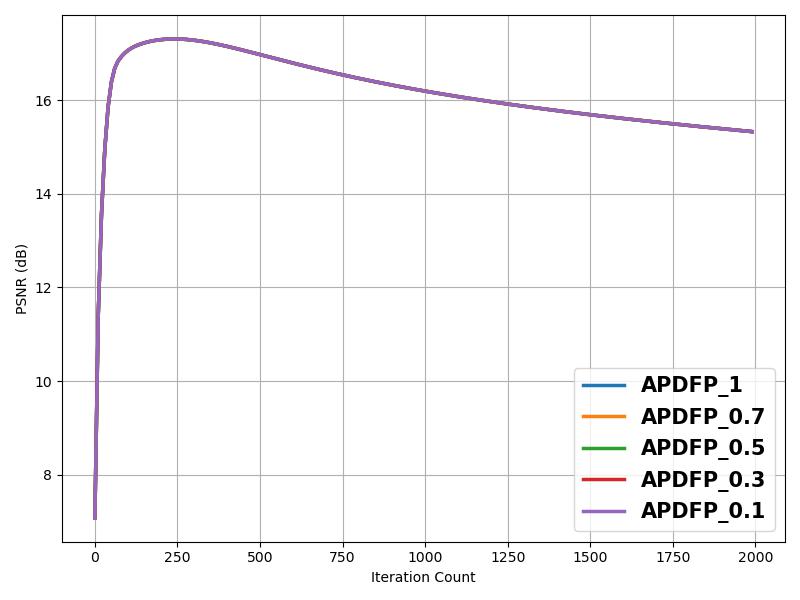}
}
\caption{The objective function value and PSNR across iterations for different values of~$\lambda$.}
\label{CT-2D-lam}
\end{figure}

\section{Conclusions.}\label{sec:conclusion}
In this work, we propose an accelerated primal--dual fixed point (APDFP) algorithm that incorporates Nesterov's acceleration technique into the primal--dual fixed point (PDFP) algoorithm to enhance convergence speed for problems in data and imaging sciences. The proposed APDFP algorithm features a simple update rule and minimal parameter tuning requirements, and can be viewed as a generalization of NAG (See Fig.~\ref{connection} for the connection between APDFP and the existing algorithms). From a theoretical perspective, we improve the convergence rate with respect to the Lipschitz constant \( L_f \) from \( \mathcal{O}(1/k) \) to \( \mathcal{O}(1/k^2) \). Numerical experiments on several examples, together with comparisons against state-of-the-art algorithms, demonstrate the correctness and effectiveness of the proposed method.
\begin{figure}[htbp!]
\centering
\includegraphics[width = 5in]{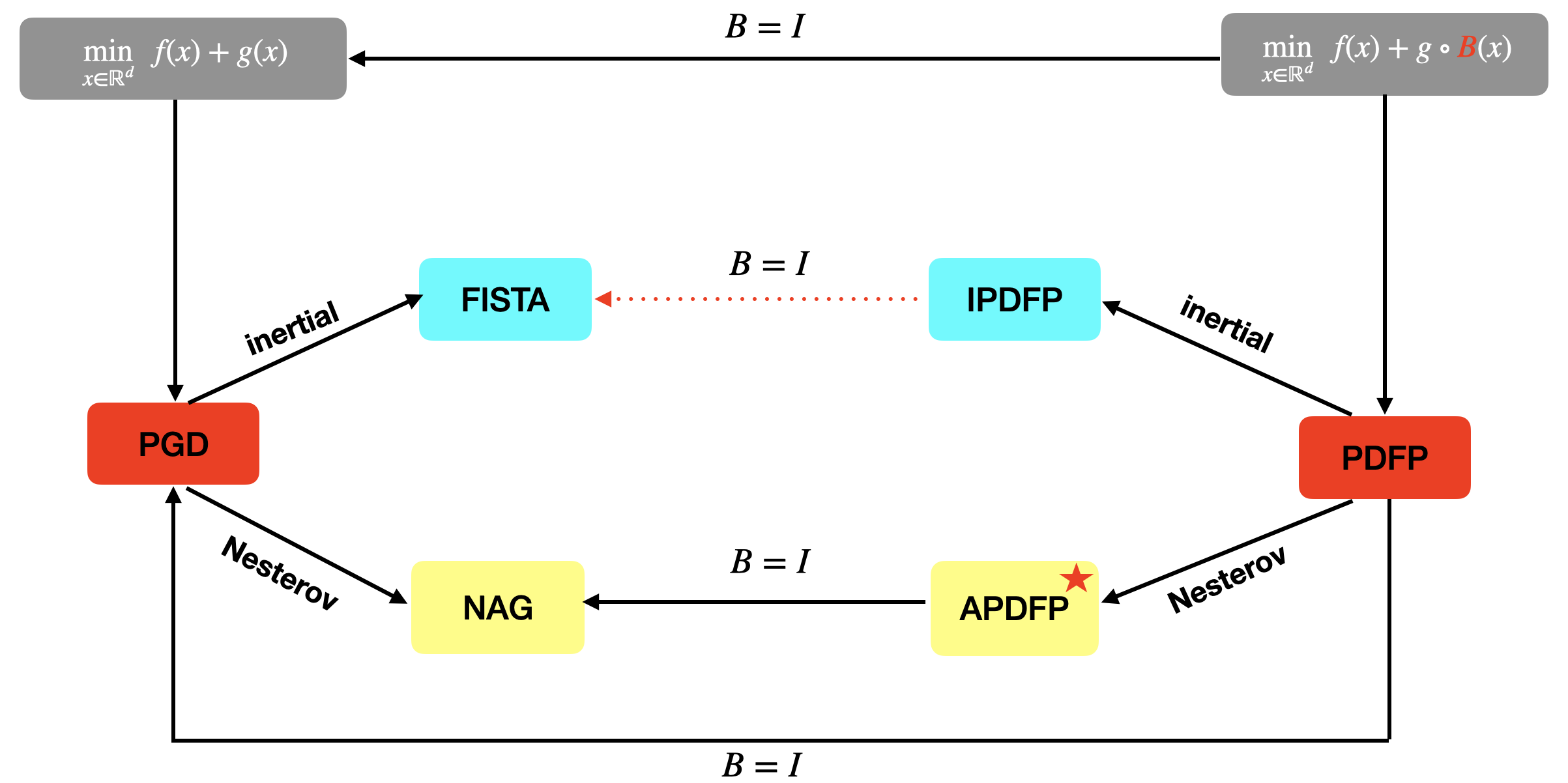}
\caption{The connection between PDFP and PGD variants. APDFP is the algorithm proposed in this work and it  generalizes NAG when the matrix $B = I$. The dash arrow between IPDFP and FISTA means that IPDFP degenerate to FISTA when $B = I$ but lack theoretical generalization of FISTA.}
\label{connection}
\end{figure}

\section{Appendix}\label{sec:appendix}
% \subsection{Appendix A} 

\begin{lemma}\label{lm1}
Let $z_{k}^{\text{ag}} = (x_{k}^{\text{ag}},y_{k}^{\text{ag}})$ be the iterate of Algorithm \ref{APDFP}, then for any $z = (x,y) \in \mathbb{R}^d \times Y$ the following equality holds
\begin{equation}\label{lm1eq1}
\begin{aligned}
& Q(z_{k+1}^{\text{ag}}, z) - (1 - \theta_k) Q(z_k^{\text{ag}}, z) \\
& = \Big[f(x_{k+1}^{\text{ag}}) - \theta_kf(x) - (1 - \theta_k)f(x_k^{\text{ag}})\Big] 
    + \Big[ g^*(y_{k+1}^{\text{ag}}) - \theta_kg^*(y) - (1 - \theta_k)g^*(y_k^{\text{ag}})\Big]  \\
&\qquad + \theta_k\langle B x_{k + 1}, y \rangle 
    - \theta_k\langle B x, y_{k + 1} \rangle \\
\end{aligned}    
\end{equation}
where $Q(\cdot,\cdot)$ is defined in the Eq. \textbf{(\ref{gap1})}.
\end{lemma}
\begin{proof}
According to the definition of $Q(\tilde{z},z)$ in Eq. \textbf{(\ref{gap1})}, it follows that 
\begin{equation}\label{pdeq21}
\begin{aligned}
& Q(z_{k+1}^{\text{ag}}, z) - (1 - \theta_k) Q(z_k^{\text{ag}}, z) \\
&=  \Big\{ [f(x_{k+1}^{\text{ag}}) + \langle B x_{k+1}^{\text{ag}}, y \rangle - g^*(y)] 
    - [f(x) + \langle B x, y_{k+1}^{\text{ag}} \rangle - g^*(y_{k+1}^{\text{ag}})] \Big\} \\
&\quad - (1 - \theta_k) \Big\{ [f(x_k^{\text{ag}}) + \langle B x_k^{\text{ag}}, y \rangle - g^*(y)] 
    - [f(x) + \langle B x, y_k^{\text{ag}} \rangle - g^*(y_k^{\text{ag}})] \Big\} \\
& = \Big[f(x_{k+1}^{\text{ag}}) - \theta_kf(x) - (1 - \theta_k)f(x_k^{\text{ag}})\Big] 
    + \Big[ g^*(y_{k+1}^{\text{ag}}) - \theta_kg^*(y) - (1 - \theta_k)g^*(y_k^{\text{ag}})\Big]  \\
&\quad + \langle B ( x_{k+1}^{\text{ag}} - (1 - \theta_k) x_k^{\text{ag}}), y \rangle 
    - \langle B x, y_{k+1}^{\text{ag}} - (1 - \theta_k) y_k^{\text{ag}} \rangle \\
& = \Big[f(x_{k+1}^{\text{ag}}) - \theta_kf(x) - (1 - \theta_k)f(x_k^{\text{ag}})\Big] 
    + \Big[ g^*(y_{k+1}^{\text{ag}}) - \theta_kg^*(y) - (1 - \theta_k)g^*(y_k^{\text{ag}})\Big]  \\
&\qquad + \theta_k\langle B x_{k + 1}, y \rangle 
    - \theta_k\langle B x, y_{k + 1} \rangle \\
\end{aligned}    
\end{equation}  
where the last equation uses the definition of $x_{k + 1}^{\text{ag}},y_{k + 1}^{\text{ag}}$ in step 6,7 of Algorithm \ref{APDFP}.
\end{proof}
The Lemma~\textbf{\ref{lm1}} provides the general form of the energy function which the accelerated algorithms aim to estimate. In the following two lemmas, we present detailed estimates for APDFP.
\begin{lemma}\label{lm2}
Suppose the function $f$ is $L_f$ smooth convex function and $g$ is a convex Lipchitz continuous. Choose the parameter $0 < \lambda \leq \frac{1}{\rho_{\max}(BB^T)}$, let $x_{k + 1}^{\text{ag}},y_{k + 1}^{\text{ag}}$ be the aggregate iterate in Algorithm \ref{APDFP}, then for any $z = (x,y) \in \mathbb{R}^d \times Y$ the following inequality holds
\begin{equation}\label{lm2eq1}
\begin{aligned}
& f(x_{k + 1}^{\text{ag}}) - (1 - \theta_k)f(x_k^{\text{ag}}) - \theta_kf(x) + g^*(y_{k + 1}^{\text{ag}}) -  (1 - \theta_k) g^*(y_{k}^{\text{ag}}) -\theta_k g^*(y) \\
& \leq \theta_k(x - x_{k+1})^TB^Ty_{k + 1} - \theta_k(y - y_{k + 1})^TB x_{k + 1} \\
& \qquad + \frac{\theta_k}{2\gamma_k}(\|x_{k} - x\|_2^2 - \|x_{k + 1} - x\|_2^2 - \|x_{k + 1} - x_{k}\|_2^2) + \frac{\theta_k^2L_f}{2}\|x_{k + 1} - x_{k}\|_2^2 \\
& \qquad + \frac{\theta_k}{2}(\|y_k - y\|_M^2 - \|y_{k + 1} - y\|_M^2 - \|y_{k + 1} - y_k\|_M^2) \\
\end{aligned}    
\end{equation}
\end{lemma}
where $M = \frac{I - \lambda BB^T}{\lambda}$.
\begin{proof}
By the $x$ update in Algorithm \ref{APDFP}, we have for any $x \in \mathbb{R}^d$, the following equation holds
\begin{equation}\label{primaleq1}
\begin{aligned}
& (x - x_{k + 1})^T(\nabla f(x_{k}^{\text{md}}) + B^Ty_{k + 1} + \frac{\theta_k}{\gamma_k}(x_{k + 1} - x_k)) = 0\\      
\Leftrightarrow 
& (\theta_k x - \theta_kx_{k + 1})^T\nabla f(x_{k}^{\text{md}}) +(\theta_k x - \theta_kx_{k + 1})^T( B^Ty_{k + 1} + \frac{\theta_k}{\gamma_k}(x_{k + 1} - x_k)) = 0\\  
\end{aligned}    
\end{equation}
The first term on the left side of the second equation of \textbf{(\ref{primaleq1})} can be estimated by
\begin{equation}\label{primaleq2}
\begin{aligned}
& (\theta_k x - \theta_kx_{k + 1})^T\nabla f(x_{k}^{\text{md}}) \\
& = (\theta_k x + (1 - \theta_k) x_k^{\text{ag}}  - ((1 - \theta_k) x_k^{\text{ag}} + \theta_kx_{k + 1}))^T\nabla f(x_{k}^{\text{md}}) \\
& \overset{\tiny{\circled{1}}}{=}  (\theta_k x + (1 - \theta_k) x_k^{\text{ag}}  - x_{k + 1}^{\text{ag}})^T\nabla f(x_{k}^{\text{md}}) \\
& \overset{\tiny{\circled{2}}}{\leq}  f(\theta_k x + (1 - \theta_k)x_k^{\text{ag}}) - f(x_{k + 1}^{\text{ag}}) + \frac{L_f}{2}\|x_{k + 1}^{\text{ag}} - x_{k}^{\text{md}}\|_2^2 \\
& \overset{\tiny{\circled{3}}}{\leq}  \theta_kf(x) + (1 - \theta_k)f(x_k^{\text{ag}}) - f(x_{k + 1}^{\text{ag}}) + \frac{L_f}{2}\|x_{k + 1}^{\text{ag}} - x_{k}^{\text{md}}\|_2^2 \\
& \overset{\tiny{\circled{4}}}{=}  \theta_kf(x) + (1 - \theta_k)f(x_k^{\text{ag}}) - f(x_{k + 1}^{\text{ag}}) + \frac{\theta_k^2L_f}{2}\|x_{k + 1} - x_{k}\|_2^2 \\
\end{aligned}    
\end{equation}
where 
\begin{itemize}
    \item $\small{\circled{1}}$ uses the definition of $x_{k + 1}^{\text{ag}}$ (step 6 of Algorithm \ref{APDFP}).
    \item $\small{\circled{2}}$ follows from convexity and smoothness of $f$ (descent lemma  \cite{Bookbeck}).
    \item $\small{\circled{3}}$ uses convexity of $f$.
    \item $\small{\circled{4}}$ uses the definition of $x_{k + 1}^{\text{ag}}$ and $x_{k}^{\text{md}}$ (step 2 and 6 of Algorithm \ref{APDFP}).
\end{itemize}
Substituting Eq. \textbf{(\ref{primaleq2})} into Eq. \textbf{(\ref{primaleq1})} yields
\begin{equation}\label{primaleq3}
\begin{aligned}
f(x_{k + 1}^{\text{ag}}) 
& \leq  \theta_kf(x) + (1 - \theta_k)f(x_k^{\text{ag}}) + \frac{\theta_k^2L_f}{2}\|x_{k + 1} - x_{k}\|_2^2+ \theta_k(x - x_{k+1})^TB^Ty_{k + 1} \\
& \qquad + \frac{\theta_k^2}{2\gamma_k}(\|x_{k} - x\|_2^2 - \|x_{k + 1} - x\|_2^2 - \|x_{k + 1} - x_{k}\|_2^2)   
\end{aligned}    
\end{equation}
By the first optimality condition for the $y$ update in Algorithm \ref{APDFP}, one has for any $y \in Y$ and $g^{'}(y_{k + 1}) \in \partial g^*(y_{k + 1})$
\begin{equation}\label{dualeq1}
\begin{aligned}
& (y - y_{k + 1})^T(\frac{\lambda}{\gamma_k}\theta_kg'(y_{k + 1})- \frac{\lambda}{\gamma_k}\theta_kB\overline{x}_{k} + (y_{k + 1} - y_k)) \geq 0, \\
\Leftrightarrow
& (y - y_{k + 1})^T(\frac{\lambda}{\gamma_k}\theta_kg'(y_{k + 1}) - \frac{\lambda}{\gamma_k}\theta_kB(x_{k + 1} + \frac{\gamma_k}{\theta_k} B^T (y_{k + 1} - y_k)) + (y_{k + 1} - y_k)) \geq 0\\
\Leftrightarrow
& (y - y_{k + 1})^T(\frac{\lambda}{\gamma_k}\theta_kg'(y_{k + 1}) - \frac{\lambda}{\gamma_k}\theta_kBx_{k + 1} + (I - \lambda BB^T)(y_{k + 1} - y_k)) \geq 0\\
\Leftrightarrow
& (\theta_ky - \theta_ky_{k + 1})^Tg'(y_{k + 1}) - \theta_k(y - y_{k + 1})^TBx_{k + 1} + \gamma_k(y - y_{k + 1})^TM(y_{k + 1} - y_k) \geq 0\\ 
\end{aligned}    
\end{equation}
where the second inequality uses the definition of $\overline{x}_{k}$ and $x_{k + 1}$ (step 3 and 5 of Algorithm \ref{APDFP}) and the last equation uses the definition $M = \frac{I - \lambda BB^T}{\lambda}$. \\
The first term of the last inequality in \textbf{(\ref{dualeq1})} is estimated by 
\begin{equation}\label{dualeq2}
\begin{aligned}
& (\theta_k y  - \theta_ky_{k + 1}))^Tg'(y_{k + 1}) \\
& \leq \theta_k g^*(y) - \theta_k g^*(y_{k + 1}) \\
& = \theta_k g^*(y) + (1 - \theta_k) g^*(y_{k}^{\text{ag}}) -( (1 - \theta_k) g^*(y_{k}^{\text{ag}}) + \theta_k g^*(y_{k + 1})) \\
& \leq \theta_k g^*(y) + (1 - \theta_k) g^*(y_{k}^{\text{ag}}) - g^*(y_{k + 1}^{\text{ag}})
\end{aligned}    
\end{equation}
where the first inequality uses the convexity of $g^*$ and the last inequality follows from the convexity of $g^*$ and the definition of $y_{k + 1}^{\textit{ag}}$ (step 7 of Algorithm \ref{APDFP}). \\
Taking Eq. \textbf{(\ref{dualeq2})} into Eq. \textbf{(\ref{dualeq1})}, one gets
\begin{equation}\label{dualeq3}
\begin{aligned}
g^*(y_{k + 1}^{\text{ag}})
& \leq \theta_k g^*(y) + (1 - \theta_k) g^*(y_{k}^{\text{ag}}) - \theta_k(y - y_{k + 1})^TBx_{k + 1} + \gamma_k(y - y_{k + 1})^TM(y_{k + 1} - y_k)) \\
& \leq \theta_k g^*(y) + (1 - \theta_k) g^*(y_{k}^{\text{ag}}) - \theta_k(y - y_{k + 1})^TBx_{k + 1} \\
& \qquad + \frac{\gamma_k}{2}(\|y - y_k\|_M^2 - \|y - y_{k + 1}\|_M^2 - \|y_k - y_{k + 1}\|_M^2) \\
\end{aligned}    
\end{equation}
Combining Eq. \textbf{(\ref{primaleq3})} and \textbf{(\ref{dualeq3})}, we arrive at
\begin{equation}\label{pdeq1}
\begin{aligned}
& f(x_{k + 1}^{\text{ag}}) - (1 - \theta_k)f(x_k^{\text{ag}}) - \theta_kf(x) + g^*(y_{k + 1}^{\text{ag}}) -  (1 - \theta_k) g^*(y_{k}^{\text{ag}}) -\theta_k g^*(y) \\
& \leq \theta_k(x - x_{k+1})^TB^Ty_{k + 1} - \theta_k(y - y_{k + 1})^TB x_{k + 1} \\
& \qquad + \frac{\theta_k^2}{2\gamma_k}(\|x_{k} - x\|_2^2 - \| x_{k + 1} - x\|_2^2 - \|x_{k + 1} - x_{k}\|_2^2) + \frac{\theta_k^2L_f}{2}\|x_{k + 1} - x_{k}\|_2^2 \\
& \qquad + \frac{\gamma_k}{2}(\|y_k - y\|_M^2 - \|y_{k + 1} - y\|_M^2 - \|y_{k + 1} - y_k\|_M^2) \\
\end{aligned}    
\end{equation}
\end{proof}

\begin{lemma}\label{lm3}
Suppose the function $f$ is $L_f$ smooth convex function and $g$ is a convex Lipchitz continuous. Choose the parameter $0 < \lambda \leq \frac{1}{\rho_{\max}(BB^T)}$, then the following estimate holds
\begin{equation}\label{lm3eq1}
\begin{aligned}
& Q(z_{k+1}^{\text{ag}}, z) - (1 - \theta_k) Q(z_k^{\text{ag}}, z) \\
& \leq \frac{\theta_k^2}{2\gamma_k}(\|x_{k} - x\|_2^2 - \|x_{k + 1} - x\|_2^2) + \frac{\theta_k^2}{2}\left(\frac{\gamma_kL_f - 1}{\gamma_k}\right)\|x_{k + 1} - x_{k}\|_2^2  \\
& \qquad + \frac{\gamma_k}{2}(\|y_k - y\|_M^2 - \|y_{k + 1} - y\|_M^2 - \|y_{k + 1} - y_k\|_M^2) \\
\end{aligned}    
\end{equation} 
where $M = \frac{I - \lambda BB^T}{\lambda}$.
\end{lemma}
\begin{proof}
Combining Lemma \textbf{(\ref{lm1})} and \textbf{(\ref{lm2})}, one has
\begin{equation}
\begin{aligned}
& Q(z_{k+1}^{\text{ag}}, z) - (1 - \theta_k) Q(z_k^{\text{ag}}, z) \\
& \overset{\tiny{\circled{1}}}{=} \Big[f(x_{k+1}^{\text{ag}}) - \theta_kf(x) - (1 - \theta_k)f(x_k^{\text{ag}})\Big] 
    + \Big[ g^*(y_{k+1}^{\text{ag}}) - \theta_kg^*(y) - (1 - \theta_k)g^*(y_k^{\text{ag}})\Big]  \\
&\qquad + \theta_k\langle B x_{k + 1}, y \rangle 
    - \theta_k\langle B x, y_{k + 1} \rangle \\
& \overset{\tiny{\circled{2}}}{\leq}\theta_k(x - x_{k+1})^TB^Ty_{k + 1} - \theta_k(y - y_{k + 1})^TB x_{k + 1} + \theta_k\langle B x_{k + 1}, y \rangle 
    - \theta_k\langle B x, y_{k + 1} \rangle \\
& \qquad + \frac{\theta_k^2}{2\gamma_k}(\|x_{k} - x\|_2^2 - \| x_{k + 1} - x\|_2^2 - \|x_{k + 1} - x_{k}\|_2^2) + \frac{\theta_k^2L_f}{2}\|x_{k + 1} - x_{k}\|_2^2 \\
& \qquad + \frac{\gamma_k}{2}(\|y_k - y\|_M^2 - \|y_{k + 1} - y\|_M^2 - \|y_{k + 1} - y_{k}\|_M^2) \\
& = \frac{\theta_k^2}{2\gamma_k}(\|x_{k} - x\|_2^2 - \|x_{k + 1} - x\|_2^2) + \frac{\theta_k^2}{2}\left(\frac{\gamma_kL_f - 1}{\gamma_k}\right)\|x_{k + 1} - x_{k}\|_2^2 \\
& \qquad + \frac{\gamma_k}{2}(\|y_k - y\|_M^2 - \|y_{k + 1} - y\|_M^2 - \|y_{k + 1} - y_k\|_M^2) \\
\end{aligned}    
\end{equation}
where $\small{\circled{1}}$  follows from \textbf{(\ref{lm1eq1})} and $\small{\circled{2}}$  uses \textbf{(\ref{lm2eq1})}.
\end{proof}

% \begin{theorem}
% Suppose the function $f$ is $L_f$ smooth convex function and $g$ is a convex Lipchitz continous choose the parameter $0 < \lambda \leq \frac{1}{\rho_{\max}(BB^T)}$,  and $\gamma_k$ parameters such that 
% \begin{equation}\label{condition}
% \left\{
% \begin{aligned}
% & 0 < \gamma_k \leq \frac{1}{L_f},0 < \lambda \leq \rho_{\max}(BB^T) \\
% & \frac{\gamma_{k + 1}}{\gamma_k} \leq \frac{\theta_{k + 1}^2}{\theta_k^2(1 - \theta_{k + 1})} \\
% & \frac{\gamma_k}{\gamma_{k + 1}}\leq \frac{1}{(1 - \theta_{k + 1})} \\
% & \frac{\gamma_k^2}{\theta_k^2} \mathrm{~is~ increasing~ and ~uniformly ~ bounded}\\
% & \theta_1 = 1
% \end{aligned}    
% \right.
% \end{equation}
% then we have
% \begin{equation}\label{estimate}
% \begin{aligned}
% & \mathcal{G}_{B_1 \times B_2}(x_{k+1}^{\text{ag}},y_{k+1}^{\text{ag}}) 
% \leq \frac{\theta_k^2}{2\gamma_k}\Omega_{1} 
% +\frac{\gamma_k}{2}\frac{\rho_{\max}(I - \lambda BB^T)}{\lambda}\Omega_{2} \\
% \end{aligned}
% \end{equation}
% where $\Omega_{1},\Omega_{2}$ are constant related to dismeter of $B_1,B_2$, respectively. 
% \end{theorem}
\textbf{We are now in position to the prove the Theorem \textbf{\ref{thm1}}}.
\begin{proof}
\textbf{Step 1:(Boundedness of iterate)} Since $0 < \gamma_k \leq \frac{1}{L_f}$ and $0 < \lambda \leq \frac{1}{\rho_{\max}(BB^T)}$ ( which means the positive semi-definite of $M$), then the Eq. \textbf{(\ref{lm3eq1})} can be estimated by
\begin{equation}
\begin{aligned}
& Q(z_{k+1}^{\text{ag}}, z) - (1 - \theta_k) Q(z_k^{\text{ag}}, z) \\
& \leq \frac{\theta_k^2}{2\gamma_k}(\|x_{k} - x\|_2^2 - \|x_{k + 1} - x\|_2^2) + \frac{\gamma_k}{2}(\|y_k - y\|_M^2 - \|y_{k + 1} - y\|_M^2) \\
\end{aligned}    
\end{equation} 
Divide both side by $\frac{\theta_k^2}{2\gamma_k}$ yields
\begin{equation}\label{thm1eq1}
\begin{aligned}
\frac{2\gamma_k}{\theta_k^2} Q(z_{k+1}^{\text{ag}}, z) 
& \leq \frac{2(1 - \theta_k)}{\theta_k^2}\gamma_k Q(z_k^{\text{ag}}, z) + \|x_{k} - x\|_2^2 - \|x_{k + 1} - x\|_2^2  \\
& \qquad + \frac{\gamma_k^2}{\theta_k^2}(\|y_k - y\|_M^2 - \|y_{k + 1} - y\|_M^2) \\
& \leq \frac{2\gamma_{k -1}}{\theta_{k -1}^2} Q(z_k^{\text{ag}}, z) + \|x_{k} - x\|_2^2 - \|x_{k + 1} - x\|_2^2  \\
& \qquad + \frac{\gamma_k^2}{\theta_k^2}(\|y_k - y\|_M^2 - \|y_{k + 1 - y}\|_M^2) \\
\end{aligned}    
\end{equation}
where the second inequality follows from the second inequality of \textbf{(\ref{condition})}, i.e.,
\begin{equation}
\begin{aligned}
\frac{\gamma_{k}}{\gamma_{k - 1}} \leq \frac{\theta_{k}^2}{\theta_{k - 1}^2(1 - \theta_{k})}
\Leftrightarrow
\frac{2(1 - \theta_k)}{\theta_k^2}\gamma_k 
\leq \frac{2\gamma_{k - 1}}{\theta_{k - 1}^2} 
\end{aligned}    
\end{equation}
Sum Eq. \textbf{(\ref{thm1eq1})} from $1$ to $k$, let $z := \hat{z} = ( \hat{x}, \hat{y} )$ be the saddle point (which means $Q(z_{k+1}^{\text{ag}}, \hat{z}) \geq 0$), and observe $\theta_1 = 1$, it follows that 
\begin{equation}\label{thm1eq2}
\begin{aligned}
\|x_{k + 1} - \hat{x}\|_2^2  
& \leq \|x_{1} - \hat{x}\|_2^2 +
\sum_{i = 1}^{k}\frac{\gamma_i^2}{\theta_i^2}(\|y_i - \hat{y}\|_M^2 - \|y_{i + 1} - \hat{y}\|_M^2) \\
& = \|x_{1} - \hat{x}\|_2^2 + \frac{\gamma_1^2}{\theta_1^2}\|y_1 - \hat{y}\|_M^2+
\sum_{i = 1}^{k - 1}\Big(\frac{\gamma_{i + 1}^2}{\theta_{i + 1}^2} - \frac{\gamma_i^2}{\theta_i^2}\Big)\|y_i - \hat{y}\|_M^2 - \frac{\gamma_{k}^2}{\theta_{k}^2}\|y_{k + 1} - \hat{y}\|_M^2 \\
\end{aligned}    
\end{equation}
Since $g$ is Lipchitz continuous the domain of $g^*$ is bounded (Corollary 13.3.3 in \cite{tyrrell1970convex}). Therefore, the $\| y_{i + 1} - \hat{y}\|_M^2$ is uniformly bounded (denote the bound as $\Omega(\hat{y})$).
Using the assumption that the $\frac{\gamma_i^2}{\theta_i^2}$ is increasing, then the Eq. \textbf{(\ref{thm1eq1})} can be rewritten as 
\begin{equation}
\begin{aligned}
\|\hat{x} - x_{k + 1}\|_2^2  
& \leq \|x_{1} - \hat{x}\|_2^2 + \frac{\gamma_1^2}{\theta_1^2}\| y_1 - \hat{y}\|_M^2+
\sum_{i = 1}^{k - 1}\left(\frac{\gamma_{i + 1}^2}{\theta_{i + 1}^2} - \frac{\gamma_i^2}{\theta_i^2}\right)\|y_i - \hat{y}\|_M^2 - \frac{\gamma_k^2}{\theta_k^2}\|y_{k + 1} - \hat{y}\|_M^2 \\
& \leq \|x_{1} - \hat{x}\|_2^2 + \frac{\gamma_1^2}{\theta_1^2}\Omega(\hat{y}) +
\sum_{i = 1}^{k - 1}\left(\frac{\gamma_{i + 1}^2}{\theta_{i + 1}^2} - \frac{\gamma_i^2}{\theta_i^2}\right)\Omega(\hat{y}) \\
& = \|x_{1} - \hat{x}\|_2^2 + \frac{\gamma_k^2}{\theta_k^2}\Omega(\hat{y}) \\
\end{aligned}    
\end{equation}
where the second inequality uses the monotonicity of $\frac{\gamma_k^2}{\theta_k^2}$. Since $\frac{\gamma_k^2}{\theta_k^2}$ is uniformly bounded, we obtain the boundedness of $x_k$. \\
\textbf{Step 2: (Convergence rate)} Firstly, define $\Upsilon_k$ recursively as follows
\begin{equation}\label{thm1eq3}
\left(\frac{1}{\theta_{k + 1}} - 1\right)\Upsilon_{k + 1} = \frac{\Upsilon_{k}}{\theta_{k}},\Upsilon_{1} = 1, k = 1,2,\cdots
\end{equation}
Multiply \textbf{(\ref{lm3eq1})} by $\frac{\Upsilon_k}{\theta_k}$ and use Eq. \textbf{(\ref{thm1eq3})}, one gets 
\begin{equation}\label{thm1eq4}
\begin{aligned}
&\left(\frac{1}{\theta_{k + 1}} - 1\right)\Upsilon_{k + 1}Q(z_{k+1}^{\text{ag}}, z) - \left(\frac{1}{\theta_k} - 1\right)\Upsilon_kQ(z_k^{\text{ag}}, z) \\
& \leq  \frac{\theta_k\Upsilon_k}{2\gamma_k}(\|x_{k} - x\|_2^2 - \|x_{k + 1} - x\|_2^2) + \frac{\Upsilon_k\theta_k}{2}\left(\frac{\gamma_kL_f - 1}{\gamma_k}\right)\|x_{k + 1} - x_{k}\|_2^2  \\
& \qquad + \frac{\gamma_k\Upsilon_k }{2\theta_k}(\|y_k - y\|_M^2 - \|y_{k + 1} - y\|_M^2 - \|y_{k + 1} - y_k\|_M^2) \\
& \leq \frac{\theta_k\Upsilon_k}{2\gamma_k}(\|x_{k} - x\|_2^2 - \| x_{k + 1} - x\|_2^2) + \frac{\gamma_k\Upsilon_k }{2\theta_k}(\| y_k - y\|_M^2 - \|y_{k + 1} - y\|_M^2) \\
\end{aligned}    
\end{equation}
where the last inequality follows from the fact $0 < \gamma_k \leq \frac{1}{L_f}$ and $0 < \lambda \leq \rho_{\max}(BB^T)$ which means the matrix $M$ is positive semi-definite. \\
Telescoping Eq. \textbf{(\ref{thm1eq4})} yields
\begin{equation}\label{thm1eq5}
\begin{aligned}
& \frac{\Upsilon_k}{\theta_k}Q(z_{k+1}^{\text{ag}}, z)\\
& = (\frac{1}{\theta_{k + 1}} - 1)\Upsilon_{k + 1}Q(z_{k+1}^{\text{ag}}, z) - (\frac{1}{\theta_1} - 1)\Upsilon_kQ(z_1^{\text{ag}}, z) \\
& \leq \frac{\theta_1\Upsilon_1}{2\gamma_1}\|x - x_{1}\|_2^2 - \frac{1}{2}\sum_{i = 1}^{k - 1}\left (\frac{\theta_i\Upsilon_i}{\gamma_i} - \frac{\theta_{i + 1}\Upsilon_{i + 1}}{\gamma_{i + 1}}\right)\|x - x_{i + 1}\|_2^2 - \frac{\theta_{k}\Upsilon_k}{\gamma_{k}}\|x - x_{k + 1}\|_2^2 \\
& \qquad +\frac{\gamma_1\Upsilon_1}{2\theta_1}\|y - y_1\|_M^2 -\frac{1}{2}\sum_{i = 1}^{k - 1}\left (\frac{\gamma_i\Upsilon_i}{\theta_i}- \frac{\gamma_{i + 1}\Upsilon_{i + 1}}{\theta_{i + 1}}\right )\|y - y_{i + 1}\|_M^2 - \frac{\gamma_{k}\Upsilon_{k}}{\theta_{k}}\|y - y_{k + 1}\|_M^2 \\
\end{aligned}    
\end{equation}
Recall the second and third inequality of \textbf{(\ref{condition})}, one gets 
\begin{equation}\label{thm1eq6}
\begin{aligned}
& \frac{\gamma_{i + 1}}{\gamma_i} \leq \frac{\theta_{i + 1}^2}{\theta_i^2(1 - \theta_{i + 1})}
\Leftrightarrow \frac{\theta_i\gamma_{i + 1}}{\theta_{i + 1}\gamma_i} \leq \frac{\theta_{i + 1}}{\theta_i(1 - \theta_{i + 1})} = \frac{\Upsilon_{i + 1}}{\Upsilon_{i}} 
\Leftrightarrow \frac{\theta_i\Upsilon_i}{\gamma_i} \leq \frac{\theta_{i + 1}\Upsilon_{i + 1}}{\gamma_{i + 1}}.\\
& \frac{\gamma_i}{\gamma_{i + 1}}\leq \frac{1}{(1 - \theta_{i + 1})} 
\Leftrightarrow
\frac{\gamma_i\theta_{i + 1}}{\gamma_{i + 1}\theta_i}\leq \frac{\theta_{i + 1}}{\theta_i(1 - \theta_{i + 1})} = \frac{\Upsilon_{i + 1}}{\Upsilon_{i}}
\Leftrightarrow \frac{\gamma_i\Upsilon_i}{\theta_i} \leq \frac{\gamma_{i + 1}\Upsilon_{i + 1}}{\theta_{i + 1}}.
\end{aligned}
\end{equation}
Since $x_i,y_i$ are uniformly bounded, for fixed $x,y$ we obtain the boundedness of $\|x_i - x\|_2^2$ and $\|y_{i} - y\|_2^2$ which are denoted by $\Omega(x)$ and $\Omega(y)$, then  \textbf{(\ref{thm1eq5})} is estimated as 
\begin{equation}
\begin{aligned}
&\frac{\Upsilon_k}{\theta_k}Q(z_{k+1}^{\text{ag}}, z) \\
& \leq \frac{\theta_1\Upsilon_1}{2\gamma_1}\Omega(x) - \frac{1}{2}\sum_{i = 1}^{k - 1}\left (\frac{\theta_i\Upsilon_i}{\gamma_i} - \frac{\theta_{i + 1}\Upsilon_{i + 1}}{\gamma_{i + 1}}\right)\Omega(x) - \frac{\theta_{k}\Upsilon_k}{\gamma_{k}}\|x - x_{k + 1}\|_2^2 \\
& \qquad +\frac{\gamma_1\Upsilon_1}{2\theta_1}\frac{\rho_{\max}(I - \lambda BB^T)}{\lambda}\Omega(y) -\frac{1}{2}\sum_{i = 1}^{k - 1}\left (\frac{\gamma_i\Upsilon_i}{\theta_i}- \frac{\gamma_{i + 1}\Upsilon_{i + 1}}{\theta_{i + 1}}\right )\frac{\rho_{\max}(I - \lambda BB^T)}{\lambda}\Omega(y) \\
& \qquad - \frac{\gamma_{k}\Upsilon_{k}}{\theta_{k}}\|y - y_{k + 1}\|_M^2 \\
& \leq \frac{\theta_k\Upsilon_k}{2\gamma_k}\Omega(x) - \frac{\theta_{k}\Upsilon_k}{\gamma_{k}}\|x - x_{k + 1}\|_2^2 
+\frac{\gamma_k\Upsilon_k}{2\theta_k}\frac{\rho_{\max}(I - \lambda BB^T)}{\lambda}\Omega(y) - \frac{\gamma_{k}\Upsilon_{k}}{\theta_{k}}\|y - y_{k + 1}\|_M^2 \\
& \leq \frac{\theta_k\Upsilon_k}{2\gamma_k}\Omega(x) 
+\frac{\gamma_k\Upsilon_k}{2\theta_k}\frac{\rho_{\max}(I - \lambda BB^T)}{\lambda}\Omega(y) \\
\Leftrightarrow 
& Q(z_{k+1}^{\text{ag}}, z)
\leq \frac{\theta_k^2}{2\gamma_k}\Omega(x) 
+\frac{\gamma_k}{2}\frac{\rho_{\max}(I - \lambda BB^T)}{\lambda}\Omega(y)
\end{aligned}
\end{equation}
where the second inequality uses the Eq. \textbf{(\ref{thm1eq6})}. \\
Taking a maximum over $z = (x,y) \in B_1 \times B_2$ yields
\begin{equation}
\begin{aligned}
\mathcal{G}_{B_1 \times B_2}(x_{k+1}^{\text{ag}},y_{k+1}^{\text{ag}})
& = \max_{z = (x,y) \in B_1 \times B_2}Q(z_{k+1}^{\text{ag}}, z)
\leq \frac{\theta_k^2}{2\gamma_k}\Omega_{1} 
+\frac{\gamma_k}{2}\frac{\rho_{\max}(I - \lambda BB^T)}{\lambda}\Omega_{2}.  
\end{aligned}
\end{equation}
where $\Omega_{1} = \max_{x \in B_1}(\Omega(x)),\Omega_{2} = \max_{y \in B_2}(\Omega(y))$ whose values are related to the  diameter of $B_1$ and $B_2$, respectively. 
\end{proof}

\textbf{Proof of Corollary \ref{col1}}.
\begin{proof}
It is clear that the $0 < \gamma_k < \frac{1}{L_f}$ and $\theta_1 = 1$. For the second inequality in \textbf{(\ref{condition})}, it follows that
\begin{equation}
\begin{aligned}
\frac{\gamma_{k + 1}}{\gamma_k}
& \leq \frac{\theta_{k + 1}^2}{\theta_k^2(1 - \theta_{k + 1})} \\
\Leftrightarrow
\frac{L_f + ck}{L_f + c(k+1)}
& \leq \frac{(k + 1)^2}{(k + 2)^2}\frac{k + 2}{k}
= \frac{(k + 1)^2}{(k + 2)k} \\
\Leftrightarrow
(L_f + ck)(k + 2)k & \leq (k + 1)^2(L_f + c(k + 1)) \\
\Leftrightarrow
L_f(k + 2)k + ck^2(k + 2) & \leq L_f(k + 1)^2 + c(k + 1)^3 \\
\Leftrightarrow
0 & \leq L_f + c(k^2 + 3k + 1) \\
\Leftrightarrow & \mathrm{trival}.\\
\end{aligned}    
\end{equation}
The third inequality in \textbf{(\ref{condition})} is equivalent to 
\begin{equation}
\begin{aligned}
\frac{\gamma_k}{\gamma_{k + 1}}
& \leq \frac{1}{(1 - \theta_{k + 1})} \\
\Leftrightarrow
\frac{L_f + c(k + 1)}{L_f + ck}
& \leq \frac{k + 2}{k} \\
\Leftrightarrow
L_f k + ck(k + 1)
& \leq (k + 2)L_f + ck(k + 2) \\
\Leftrightarrow
0
& \leq 2L_f + ck \\
\Leftrightarrow & \mathrm{trival}.\\
\end{aligned}    
\end{equation}
Let us prove the increasing property of  $\frac{\gamma_k^2}{\theta_k^2} = \frac{(k + 1)^2}{4(L_f + ck)^2}$.  \\
Define $h(x) =  \frac{(x + 1)^2}{4(L_f + cx)^2}$. Taking the derivative of $h$ yields,
\begin{equation}
h'(x) = \frac{(x + 1)(L_f - c)}{2(L_f + c x)^3}    
\end{equation}
Since $L_f > c  \rightarrow  h^{'}(x) > 0$, therefore $\frac{\gamma_k^2}{\theta_k^2}$ is increasing. It is readily to verify that the limit of $\frac{\gamma_k^2}{\theta_k^2}$ exists as $k \to \infty$, therefore it is bounded.\\
Thus, the parameters $\theta_k,\gamma_k$ satisfies the condition in \textbf{(\ref{condition})}. \\
Plug the parameters in the Eq. \textbf{(\ref{estimate})}, we obtain the results.
\end{proof}

\normalem
\bibliographystyle{siam}
\bibliography{ref}

\begin{thebibliography}{10}

\bibitem{PDHG}
{\sc K.~J. Arrow, L.~Hurwicz, H.~Uzawa, H.~B. Chenery, S.~Johnson, and S.~Karlin}, {\em Studies in linear and non-linear programming}, vol.~2, Stanford University Press Stanford, 1958.

\bibitem{GLasso}
{\sc O.~Banerjee, L.~El~Ghaoui, and A.~d'Aspremont}, {\em Model selection through sparse maximum likelihood estimation for multivariate gaussian or binary data}, The Journal of Machine Learning Research, 9 (2008), pp.~485--516.

\bibitem{cvxhl}
{\sc H.~H. Bauschke and P.~L. Combettes}, {\em Correction to: convex analysis and monotone operator theory in hilbert spaces}, in Convex analysis and monotone operator theory in Hilbert spaces, Springer, 2020, pp.~C1--C4.

\bibitem{Monotone}
{\sc H.~H. Bauschke, P.~L. Combettes, H.~H. Bauschke, and P.~L. Combettes}, {\em Correction to: convex analysis and monotone operator theory in Hilbert spaces}, Springer, 2017.

\bibitem{Bookbeck}
{\sc A.~Beck}, {\em First-order methods in optimization}, SIAM, 2017.

\bibitem{FISTA}
{\sc A.~Beck and M.~Teboulle}, {\em A fast iterative shrinkage-thresholding algorithm for linear inverse problems}, SIAM journal on imaging sciences, 2 (2009), pp.~183--202.

\bibitem{ADMM3}
{\sc S.~Boyd, N.~Parikh, E.~Chu, B.~Peleato, J.~Eckstein, et~al.}, {\em Distributed optimization and statistical learning via the alternating direction method of multipliers}, Foundations and Trends{\textregistered} in Machine learning, 3 (2011), pp.~1--122.

\bibitem{lowrank}
{\sc J.-F. Cai, E.~J. Cand{\`e}s, and Z.~Shen}, {\em A singular value thresholding algorithm for matrix completion}, SIAM Journal on optimization, 20 (2010), pp.~1956--1982.

\bibitem{cs2}
{\sc E.~J. Cand{\`e}s, J.~Romberg, and T.~Tao}, {\em Robust uncertainty principles: Exact signal reconstruction from highly incomplete frequency information}, IEEE Transactions on information theory, 52 (2006), pp.~489--509.

\bibitem{cs3}
{\sc E.~J. Cand{\`e}s and M.~B. Wakin}, {\em An introduction to compressive sampling}, IEEE signal processing magazine, 25 (2008), pp.~21--30.

\bibitem{CP}
{\sc A.~Chambolle and T.~Pock}, {\em A first-order primal-dual algorithm for convex problems with applications to imaging}, Journal of mathematical imaging and vision, 40 (2011), pp.~120--145.

\bibitem{LIBSVM}
{\sc C.-C. Chang and C.-J. Lin}, {\em Libsvm: a library for support vector machines}, ACM transactions on intelligent systems and technology (TIST), 2 (2011), pp.~1--27.

\bibitem{PDFP}
{\sc P.~Chen, J.~Huang, and X.~Zhang}, {\em A primal--dual fixed point algorithm for convex separable minimization with applications to image restoration}, Inverse Problems, 29 (2013), p.~025011.

\bibitem{LADMM1}
{\sc Y.~Chen, W.~Hager, F.~Huang, D.~Phan, X.~Ye, and W.~Yin}, {\em Fast algorithms for image reconstruction with application to partially parallel mr imaging}, SIAM Journal on Imaging Sciences, 5 (2012), pp.~90--118.

\bibitem{APD}
{\sc Y.~Chen, G.~Lan, and Y.~Ouyang}, {\em Optimal primal-dual methods for a class of saddle point problems}, SIAM Journal on Optimization, 24 (2014), pp.~1779--1814.

\bibitem{PGD}
{\sc P.~L. Combettes and V.~R. Wajs}, {\em Signal recovery by proximal forward-backward splitting}, Multiscale modeling \& simulation, 4 (2005), pp.~1168--1200.

\bibitem{condatreview}
{\sc L.~Condat, D.~Kitahara, A.~Contreras, and A.~Hirabayashi}, {\em Proximal splitting algorithms for convex optimization: A tour of recent advances, with new twists}, SIAM Review, 65 (2023), pp.~375--435.

\bibitem{cs1}
{\sc D.~L. Donoho}, {\em Compressed sensing}, IEEE Transactions on information theory, 52 (2006), pp.~1289--1306.

\bibitem{PAPC}
{\sc Y.~Drori, S.~Sabach, and M.~Teboulle}, {\em A simple algorithm for a class of nonsmooth convex--concave saddle-point problems}, Operations Research Letters, 43 (2015), pp.~209--214.

\bibitem{PDHGm}
{\sc E.~Esser, X.~Zhang, and T.~F. Chan}, {\em A general framework for a class of first order primal-dual algorithms for convex optimization in imaging science}, SIAM Journal on Imaging Sciences, 3 (2010), pp.~1015--1046.

\bibitem{ADMM1}
{\sc D.~Gabay and B.~Mercier}, {\em A dual algorithm for the solution of nonlinear variational problems via finite element approximation}, Computers \& mathematics with applications, 2 (1976), pp.~17--40.

\bibitem{Xray}
{\sc H.~Gao}, {\em Fast parallel algorithms for the x-ray transform and its adjoint}, Medical physics, 39 (2012), pp.~7110--7120.

\bibitem{ADMM2}
{\sc R.~Glowinski and A.~Marroco}, {\em Sur l'approximation, par {\'e}l{\'e}ments finis d'ordre un, et la r{\'e}solution, par p{\'e}nalisation-dualit{\'e} d'une classe de probl{\`e}mes de dirichlet non lin{\'e}aires}, Revue fran{\c{c}}aise d'automatique, informatique, recherche op{\'e}rationnelle. Analyse num{\'e}rique, 9 (1975), pp.~41--76.

\bibitem{goldstein2009split}
{\sc T.~Goldstein and S.~Osher}, {\em The split bregman method for l1-regularized problems}, SIAM journal on imaging sciences, 2 (2009), pp.~323--343.

\bibitem{golub2013matrix}
{\sc G.~H. Golub and C.~F. Van~Loan}, {\em Matrix computations}, JHU press, 2013.

\bibitem{BPG}
{\sc F.~Hanzely, P.~Richtarik, and L.~Xiao}, {\em Accelerated bregman proximal gradient methods for relatively smooth convex optimization}, Computational Optimization and Applications, 79 (2021), pp.~405--440.

\bibitem{GL1}
{\sc L.~Jacob, G.~Obozinski, and J.-P. Vert}, {\em Group lasso with overlap and graph lasso}, in Proceedings of the 26th annual international conference on machine learning, 2009, pp.~433--440.

\bibitem{Booklan}
{\sc G.~Lan}, {\em First-order and stochastic optimization methods for machine learning}, vol.~1, Springer, 2020.

\bibitem{sc2}
{\sc H.~Lee, A.~Battle, R.~Raina, and A.~Ng}, {\em Efficient sparse coding algorithms}, Advances in neural information processing systems, 19 (2006).

\bibitem{LR}
{\sc I.~Loris and C.~Verhoeven}, {\em On a generalization of the iterative soft-thresholding algorithm for the case of non-separable penalty}, Inverse Problems, 27 (2011), p.~125007.

\bibitem{FP2O}
{\sc C.~A. Micchelli, L.~Shen, and Y.~Xu}, {\em Proximity algorithms for image models: denoising}, Inverse Problems, 27 (2011), p.~045009.

\bibitem{monteiro2010complexity}
{\sc R.~D. Monteiro and B.~F. Svaiter}, {\em On the complexity of the hybrid proximal extragradient method for the iterates and the ergodic mean}, SIAM Journal on Optimization, 20 (2010), pp.~2755--2787.

\bibitem{monteiro2011complexity}
\leavevmode\vrule height 2pt depth -1.6pt width 23pt, {\em Complexity of variants of tseng's modified fb splitting and korpelevich's methods for hemivariational inequalities with applications to saddle-point and convex optimization problems}, SIAM Journal on Optimization, 21 (2011), pp.~1688--1720.

\bibitem{monteiro2013iteration}
{\sc R.~D. Monteiro and B.~F. Svaiter}, {\em Iteration-complexity of block-decomposition algorithms and the alternating direction method of multipliers}, SIAM Journal on Optimization, 23 (2013), pp.~475--507.

\bibitem{nesterov1}
{\sc Y.~Nesterov}, {\em A method for unconstrained convex minimization problem with the rate of convergence o (1/k2)}, in Dokl. Akad. Nauk. SSSR, vol.~269, 1983, p.~543.

\bibitem{nesterov2}
{\sc Y.~Nesterov}, {\em Smooth minimization of non-smooth functions}, Mathematical programming, 103 (2005), pp.~127--152.

\bibitem{Nes2nd}
\leavevmode\vrule height 2pt depth -1.6pt width 23pt, {\em Gradient methods for minimizing composite functions}, Mathematical programming, 140 (2013), pp.~125--161.

\bibitem{Booknesterov}
{\sc Y.~Nesterov}, {\em Introductory lectures on convex optimization: A basic course}, vol.~87, Springer Science \& Business Media, 2013.

\bibitem{Nes83}
{\sc Y.~E. Nesterov}, {\em A method for solving the convex programming problem with convergence rate \( o(1/k^2) \)}, Doklady Akademii Nauk SSSR, 269 (1983), pp.~543--547.

\bibitem{sc1}
{\sc B.~A. Olshausen and D.~J. Field}, {\em Sparse coding with an overcomplete basis set: A strategy employed by v1?}, Vision research, 37 (1997), pp.~3311--3325.

\bibitem{AADMM}
{\sc Y.~Ouyang, Y.~Chen, G.~Lan, and E.~Pasiliao~Jr}, {\em An accelerated linearized alternating direction method of multipliers}, SIAM Journal on Imaging Sciences, 8 (2015), pp.~644--681.

\bibitem{TV}
{\sc L.~I. Rudin, S.~Osher, and E.~Fatemi}, {\em Nonlinear total variation based noise removal algorithms}, Physica D: nonlinear phenomena, 60 (1992), pp.~259--268.

\bibitem{Siddon}
{\sc R.~L. Siddon}, {\em Fast calculation of the exact radiological path for a three-dimensional ct array}, Medical physics, 12 (1985), pp.~252--255.

\bibitem{GL2}
{\sc R.~Tibshirani, M.~Saunders, S.~Rosset, J.~Zhu, and K.~Knight}, {\em Sparsity and smoothness via the fused lasso}, Journal of the Royal Statistical Society Series B: Statistical Methodology, 67 (2005), pp.~91--108.

\bibitem{tyrrell1970convex}
{\sc R.~Tyrrell~Rockafellar}, {\em Convex analysis}, Princeton mathematical series, 28 (1970).

\bibitem{LADMM2}
{\sc X.~Ye, Y.~Chen, and F.~Huang}, {\em Computational acceleration for mr image reconstruction in partially parallel imaging}, IEEE transactions on medical imaging, 30 (2010), pp.~1055--1063.

\bibitem{LADMM3}
{\sc X.~Ye, Y.~Chen, W.~Lin, and F.~Huang}, {\em Fast mr image reconstruction for partially parallel imaging with arbitrary $ k $-space trajectories}, IEEE Transactions on Medical Imaging, 30 (2010), pp.~575--585.

\bibitem{OriPDHG}
{\sc M.~Zhu and T.~Chan}, {\em An efficient primal-dual hybrid gradient algorithm for total variation image restoration}, Ucla Cam Report, 34 (2008).

\bibitem{SVRGPDFP}
{\sc Y.-N. Zhu and X.~Zhang}, {\em A stochastic variance reduced primal dual fixed point method for linearly constrained separable optimization}, SIAM Journal on Imaging Sciences, 14 (2021), pp.~1326--1353.

\bibitem{IPDFP}
{\sc Y.-N. Zhu and X.~Zhang}, {\em Two modified schemes for the primal dual fixed point method}, CSIAM Trans. Appl. Math, 2 (2021), pp.~108--130.

\end{thebibliography}
\end{document}